\documentclass{amsart}
\usepackage{amssymb,amsxtra}
\usepackage{graphicx}
\usepackage{amscd}
\usepackage{amsmath}
\usepackage{amsfonts}
\usepackage{amssymb}
\usepackage{upref}
\usepackage[all]{xy}
\usepackage{color}

\newtheorem{theorem}{Theorem}[section]
\newtheorem{proposition}[theorem]{Proposition}
\newtheorem{lemma}[theorem]{Lemma}
\newtheorem{corollary}[theorem]{Corollary}
\newtheorem{thm}[theorem]{Theorem}
\newtheorem{cor}[theorem]{Corollary}
\newtheorem{lem}[theorem]{Lemma}

\theoremstyle{definition}
\newtheorem{definition}[theorem]{Definition}
\newtheorem{defn}[theorem]{Definition}
\newtheorem{example}[theorem]{Example}

\newtheorem{remark}[theorem]{Remark}

\newtheorem{step}{Step}

\newcommand{\tht}{\theta}
\renewcommand{\bar}{\overline}
\newcommand{\<}{\langle}
\renewcommand{\>}{\rangle}
\newcommand{\id}{\text{\textup{id}}}

\renewcommand{\)}{\textup)}
\newcommand{\inv}{^{-1}}
\newcommand{\rg}{\text{rg}}
\newcommand{\midtext}[1]{\quad\text{#1}\quad}
\newcommand{\smtx}[1]{\left(\begin{smallmatrix} #1
\end{smallmatrix}\right)}
\DeclareMathOperator*{\spn}{span}
\DeclareMathOperator*{\clspn}{\overline{\spn}}
\DeclareMathOperator{\aut}{Aut}
\DeclareMathOperator{\supp}{supp}
\DeclareMathOperator{\fix}{Fix}
\DeclareMathOperator{\Orb}{Orb}
\newcommand{\R}{\mathbb R}
\newcommand{\Z}{\mathbb Z}
\newcommand{\C}{\mathbb C}
\newcommand{\N}{\mathbb N}
\newcommand{\T}{\mathbb T}
\newcommand{\KK}{\mathcal K}
\newcommand{\LL}{\mathcal L}

\newcommand{\HH}{\mathcal H}

\newcommand{\Chi}{\raisebox{2pt}{\ensuremath{\chi}}}
\newcommand{\minus}{\setminus}

\title[Group actions on topological graphs]{Group actions on topological graphs}
\author{Valentin Deaconu}
\address{Department of Mathematics\\ University of Nevada\\ Reno NV
89557-0084, USA}
\email[Valentin Deaconu]{vdeaconu@unr.edu}
\author{Alex Kumjian}
\address{Department of Mathematics\\ University of Nevada\\ Reno NV
89557-0084, USA}
\email[Alex Kumjian]{alex@unr.edu}
\author{John Quigg}
\address{School of Mathematical and Statistical Sciences\\Arizona State University\\Tempe AZ 85287-1804, USA}
\email[John Quigg]{quigg@asu.edu}

\date{\today}
\begin{document}

\begin{abstract}

We define the action of a locally compact group $G$ on a topological graph $E$. This action induces a natural action of $G$ on the $C^*$-correspondence ${\mathcal H}(E)$ and on the graph $C^*$-algebra $C^*(E)$. If the action is free and proper, we prove that $C^*(E)\rtimes_r G$ is strongly Morita equivalent to $C^*(E/G)$. We define the skew product of a locally compact group $G$ by a topological graph $E$ via a cocycle $c:E^1\to G$. The group acts freely and properly on this new topological graph $E\times_cG$. If $G$ is abelian, there is a dual action on $C^*(E)$ such that $C^*(E)\rtimes \hat{G}\cong C^*(E\times_cG)$. We also define the fundamental group and the universal covering of a topological graph.

\end{abstract}

\maketitle

\section{Introduction}

Topological graphs  generalize discrete directed graphs.  They have a richer structure, since the vertex and edge 
spaces may be arbitrary locally compact spaces, and  it is natural to study their symmetries. 
The corresponding $C^*$-algebras  provide a wealth of  intriguing examples. Katsura proved in \cite{Ka4} that all 
Kirchberg algebras can be obtained from topological graphs.  

In \cite{KP} the authors consider free actions of discrete groups $G$ on directed graphs $E$ and on their associated $C^*$-algebras $C^*(E)$. They prove that the crossed product $C^*(E)\rtimes G$ is strongly Morita equivalent to $C^*(E/G)$, where $E/G$ is the quotient graph. They also consider the notion of skew product graph associated to a cocycle $c:E^1\to G$ and the universal covering tree of a graph.

In this paper, we prove an analogous result for the reduced crossed product for $G$ a locally 
compact group acting freely and properly on a topological graph $E$.  Unlike \cite{KP}, where 
groupoids were used extensively, the $C^*$-algebra $C^*(E)$ is defined here using a $C^*$-
correspondence $\HH(E)$ and the Cuntz-Pimsner construction (see \cite{Ka1}). The main ingredients of the 
proof are a structure theorem of Palais about principal $G$-bundles, the generalized fixed point algebras of Rieffel, and the use of multiplier bimodules introduced by Echterhoff, Raeburn, and others. We need multiplier bimodules because we must construct a homomorphism of $C^*(E/G)$ into the multiplier algebra $M(C^*(E))$, and as an intermediate step we map the correspondence $\HH(E/G)$ into the multipliers of the correspondence $\HH(E)$.

We also define the notion of fundamental group of a topological graph and the universal covering, using a space $R(E)$, called the geometric realization, which is a kind of double mapping torus. In the discrete case, the geometric realization is the usual $1$-dimensional CW-complex that we associate with
a graph, and the fundamental group is free. For a topological graph $E$, the space $R(E)$ could be a higher dimensional CW-complex, and $\pi_1(E)$ is not necessarily free, as we will see in our examples.

Our paper is organized as follows. In section~2 we  recall the notion of topological graph and the construction of its $C^*$-algebra as a Cuntz-Pimsner algebra, giving several examples. In section~3 we define group actions on topological graphs, and prove a structure theorem about topological graphs on which a locally compact group acts freely and properly.  The fundamental group of a topological graph and the corresponding universal covering are defined in section~4; we also consider a number of examples. In section~5 we discuss group actions on $C^*$-correspondences, recall the notion of generalized fixed point algebra, and prove the main theorem stating that $C^*(E)\rtimes_rG$ and $C^*(E/G)$ are strongly Morita equivalent for a free and proper action. The appendix contains basic results concerning multipliers of correspondences which are needed in
section~5.

The authors would like to thank the referee for a number of helpful comments.

\section{Topological graphs and their $C^*$-algebras}

Let $E=(E^0,E^1,s,r)$ be a  topological graph. Recall that $E^0$ and $E^1$ are locally compact (Hausdorff) spaces, and that $s,r:E^1\to E^0$ are continuous maps such that $s$ is a local homeomorphism. We think of points in $E^0$ as vertices, and of points  $e\in E^1$ as edges from $s(e)$ to $r(e)$. 
A {\em path} of length $n$ in $E=(E^0,E^1,s,r)$ is a sequence $e_1e_2\cdots e_n$ of edges such that $s(e_i)=r(e_{i+1})$ for all $i$. We denote by $E^n$ the set of paths of length $n$. The space of infinite paths is denoted $E^\infty$. The maps $s$ and $r$ extend naturally to $E^n$, and $(E^0, E^n, s,r)$ becomes a topological graph. A vertex $v\in E^0$ is viewed as a path of length $0$. 

We define two open subsets $E^0_{\text{sce}}, E^0_{\text{fin}}$ of $E^0$ by
\[E^0_{\text{sce}}=\{v\in E^0: v \;\;\text{has a neighborhood}\;\; V \;\;\text{such that}\;\; r^{-1}(V)=\emptyset\},\]
\[E^0_{\text{fin}}=\{v\in E^0: v \;\;\text{has a neighborhood}\;\; V \;\; \text{such that}\;\; r^{-1}(V)\;\; \text{is compact}\},\]
and the set of regular vertices by $E^0_{\rg}:=E^0_{\text{fin}}\setminus\overline{E^0_{\text{sce}}}$. 
It is proved in Proposition~2.8 of \cite{Ka1} that $v\in E^0_{\text{rg}}$ if and only if $v$ has a neighborhood such that $r^{-1}(V)$ is compact and $r(r^{-1}(V))=V$. In particular, if $r$ is proper and surjective, we have $E^0_{\text{rg}}=E^0$.


The  $C^*$-algebra $C^*(E)$ of a topological graph is defined as the Cuntz-Pimsner algebra ${\mathcal O}_{\mathcal H}$ of the $C^*$-correspondence ${\mathcal H}={\mathcal H}(E)$ over the $C^*$-algebra $A=C_0(E^0)$, which is obtained as a completion of $C_c(E^1)$ using the inner product 
\[\langle \xi,\eta\rangle(v)=\sum_{s(e)=v}\overline{\xi(e)}\eta(e),\; \xi,\eta\in C_c(E^1)\]
and the multiplications
\[(\xi\cdot f)(e)=\xi(e)f(s(e)),\; (f\cdot\xi)(e)=f(r(e))\xi(e).\] 
The Hilbert $C_0(E^0)$-module ${\mathcal H}$ can be identified with $\{\xi\in C_0(E^1)\mid \langle \xi, \xi\rangle \in C_0(E^0)\}$.
There exists an injective $*$-homomorphism $\pi_E:C_b(E^1)\to {\mathcal L}({\mathcal H})$ given by $(\pi_E(f)\xi)(e)=f(e)\xi(e)$ such that $\pi_E(f)\in {\mathcal K}({\mathcal H})$ if and only if $f\in C_0(E^1)$. 

For more details, see \cite{Ka1}.

Recall that a {\em Toeplitz representation} of a $C^*$-correspondence ${\mathcal H}$ over $A$ in a $C^*$-algebra $C$ is a pair $(\tau,\pi)$ with $\tau:{\mathcal H}\rightarrow C$ a linear map and $\pi:A\rightarrow C$ a $*$-homomorphism, such that 
\[\tau(\xi a)=\tau(\xi)\pi(a), \;\tau(\xi)^*\tau(\eta)=\pi(\langle \xi,\eta\rangle), \;\tau(\varphi(a)\xi)=\pi(a)\tau(\xi),\] where $\varphi:A\rightarrow {\mathcal L}({\mathcal H})$ is the left multiplication.
The corresponding universal $C^*$-algebra is called the Toeplitz algebra  ${\mathcal T}_{\mathcal H}$.
A representation $(\tau,\pi)$ is {\em  covariant} if $\pi(a)=\tau^{(1)}(\varphi(a))$ for all $a$ in the ideal $\varphi^{-1}({\mathcal K}({\mathcal H}))\cap (\ker \varphi)^\perp,$  where $\tau^{(1)}:{\mathcal K}({\mathcal H})\rightarrow C$ is such that $\tau^{(1)}(\theta_{\xi,\eta})=\tau(\xi)\tau(\eta)^*.$
The Cuntz-Pimsner algebra ${\mathcal O}_{\mathcal H}$ is universal with respect to covariant representations, and it is a quotient of ${\mathcal T}_{\mathcal H}$. 
Let $(k_{\mathcal H},k_A)$ denote the canonical Toeplitz representation of the $A$-correspondence 
$\mathcal H$ in ${\mathcal O}_{\mathcal H}$.
A covariant representation $(\tau, \pi)$ of ${\mathcal H}$ in a $C^*$-algebra $C$ gives rise to a $*$-homomorphism $\tau\times\pi:{\mathcal O}_{\mathcal H}\to C$ such that
$\tau = (\tau\times\pi) \circ k_{\mathcal H}$ and $\pi =  (\tau\times\pi) \circ k_A$.

\remark
\label{regular}
 In the case of a $C^*$-correspondence ${\mathcal H}(E)$ associated with a topological graph as above, it is proved in Proposition~1.24 of \cite{Ka1} that $\ker \varphi=C_0(E^0_{\text{sce}})$ and that $\varphi^{-1}({\mathcal K}(H))=C_0(E^0_{\text{fin}})$, therefore the ideal $\varphi^{-1}({\mathcal K}({\mathcal H}))\cap (\ker \varphi)^\perp$ coincides with $C_0(E^0_{\rg})$.

\example If the vertex space $E^0$ is discrete, then the edge space $E^1$ is also discrete, and $E=(E^0, E^1, s,r)$ is an usual (discrete) graph. 
For a discrete graph $E$, its $C^*$-algebra $C^*(E)$ was initially defined as the universal $C^*$-algebra generated by mutually orthogonal projections $p_v$ for $v\in E^0$ and partial isometries $t_e$ for $e\in E^1$ with orthogonal ranges such that $t_e^*t_e=p_{s(e)}, \;t_et_e^*\leq p_{r(e)}$ and 
\begin{equation}\label{ck}
p_v=\sum_{r(e)=v}t_et_e^*\;\mbox{if}\;\;0<|r^{-1}(v)|<\infty.
\end{equation}
If ${\mathcal H}={\mathcal H}(E)$ is the associated $C^*$-correspondence described above, then
\[\varphi^{-1}({\mathcal K}(H))=C_0(\{v\in E^0: |r^{-1}(v)|<\infty\}),\]\[\ker(\varphi)=C_0(\{v\in E^0: |r^{-1}(v)|=0\}),\]
since $E^0_{\text{fin}}=\{v\in E^0: r^{-1}(v)\;\; \text{is finite}\}$ and $E^0_{\text{sce}}=\{v\in E^0:r^{-1}(v)=\emptyset\}$ is the set of sources.
The maps
\[\pi(f)=\sum_{v\in E^0}f(v)p_v, \;f\in C_0(E^0),\quad
\tau(\xi)=\sum_{e\in E^1}\xi(e)t_e, \;\xi\in C_c(E^1)\] 
give  a Toeplitz representation of $\mathcal H$ into $C^*(E)$ if and only if $t_e^*t_e=p_{s(e)}$ and  $t_et_e^*\leq p_{r(e)}$, which is covariant if and only if (\ref{ck}) is satisfied. 
Hence, the $C^*$-algebra of a discrete graph $E$, defined using generators and relations, is isomorphic to the Cuntz-Pimsner algebra ${\mathcal O}_{\mathcal H}$ (see
\cite[Proposition~12]{FLR}) and \cite[Example~1]{Ka1}).

\example  Let $E^0=E^1={\mathbb T}, s(z)=z$, and  $r(z)=e^{2\pi i\theta}z$ for $\theta\in[0,1]$ irrational. Then $C^*(E)\cong A_{\theta}$, the irrational rotation algebra. More generally, let $E^0=E^1=X$, where $X$ is a locally compact metric space, let $s=\id$ and let $r=h$ for $h:X\to X$ a homeomorphism. Then $C^*(E)\cong C_0(X)\rtimes{\mathbb Z}$, since $C^*(E)$ is the universal $C^*$-algebra generated by a copy of $C_0(X)$ and a unitary $u$ satisfying $\hat{h}(f)=u^*fu$ for $f\in C_0(X)$, where $\hat{h}(f)=f\circ h$.

\example Let $n\in{\mathbb N}\setminus\{0\}$ and let $m\in {\mathbb Z}\setminus\{0\}$. Take $E^0=E^1={\mathbb T}, s(z)=z^n, r(z)=z^m$. We get a topological graph with both $s$ and $r$ local homeomorphisms. When $m\notin n{\mathbb Z}$, $C^*(E)$ is simple and purely infinite, see \cite[Example~A.6]{Ka4}.   

\example \label{cayley}
({\em The Cayley graph of a finitely generated locally compact group}).
A locally compact group $G$ is called finitely generated if there is a finite subset $S$ of $G$  such that the closure of the subgroup generated by $S$ is all $G$.
Given a set of generators $S=\{h_1,h_2,...,h_n\}$,  define the associated Cayley graph 
$E=E(G,S)$ with $E^0=G, E^1=S\times G, s(h,g)=g$, and $r(h,g)=gh$. Note that $E(G,S)$ becomes a topological graph and that both $s$ and $r$ are local homeomorphisms. For $G$ discrete and finitely generated, we recover the usual notion of Cayley graph. The Cayley graph may change if we change the set of generators.

For $G=({\mathbb R}, +)$ and $S=\{1,\theta\}$, where $\theta$ is irrational, the corresponding Cayley graph $E$ has  $E^0={\mathbb R}$, 
$E^1=\{1,\theta\}\times{\mathbb R}$, $s(1,x)=x, r(1,x)=x+1, s(\theta, x)=x$, and
$r(\theta, x)=x+\theta$.  
Its $C^*$-algebra is simple and, if  $\theta<0$,  it is purely infinite; moreover, 
$C^*(E) \cong {\mathcal O}_2\rtimes_{\alpha}{\mathbb R}$, where $\alpha_t(S_1)=e^{it}S_1$,
$\alpha_t(S_2)=e^{it\theta}S_2$ for $t\in {\mathbb R}$ and for the standard generators $S_1,S_2$ of the Cuntz algebra ${\mathcal O}_2$ (see Theorem~2, section~3 in \cite{KK} and  \cite[Proposition 4.3]{Ka5}).

The above example should be regarded as a special case of the following.

\example ({\em Skew products of topological graphs}).
Let $E=(E^0,E^1,s,r)$ be a  topological graph, let $G$ be a locally compact group, and let $c:E^1\to G$ be a continuous function. We
call $c$ a \emph{cocycle}, and define the {\em skew product graph}
$E\times_cG=(E^0\times G, E^1\times G, \tilde s, \tilde r)$ (cf.\ \cite{KP,R}), where
\[\tilde s(e,g)=(s(e),g), \;\; \tilde r(e, g)=(r(e), gc(e)).\]
Notice that $E\times_cG$ becomes a topological graph using the product topology, since $\tilde s$ is a local homeomorphism and $\tilde r$ is continuous. The situation of a graph with one vertex and $n$ loops $\{e_1,...,e_n\}$ and a set of generators $S=\{h_1,...,h_n\}$ of a group $G$ such that $c(e_i)=h_i, i=1,...,n$ gives the Cayley graph $E(G,S)$ as a skew product.

In particular, if $E$ is the graph with one vertex and two edges $\{e,f\}$ and $G=({\mathbb R},+)$, the map $c:\{e,f\}\to {\mathbb R},\; c(e)=1, c(f)=\theta$ for $\theta$ irrational determines the topological graph described in the above example. 

If $E$ is an arbitrary discrete graph and $G={\mathbb Z}$ with $c(e)=1$ for all $e\in E^1$, 
then $E\times_cG$ is the product graph $E\times Z$, where $Z^0=Z^1={\mathbb Z}$, 
with $s(k)=k, r(k)=k+1$. There is an isomorphism 
$C^*(E\times_c G)\cong C^*(E)\rtimes{\mathbb T}$ (see \cite[Proposition 2.6]{KP}).

\section{Group actions on topological graphs}

\definition Let $E, F$ be two topological graphs. A graph morphism $\varphi:E\to F$ is a pair of continuous  maps $\varphi=(\varphi^0,\varphi^1)$ where $\varphi^i:E^i\to F^i, i=0,1$ such that $\varphi^0\circ r=r\circ \varphi^1$ and $\varphi^0\circ s=s\circ \varphi^1$, i.e. the diagram
\[\begin{CD}E^0@<s<<E^1@>r>>E^0\\@V\varphi^0VV@V\varphi^1VV@V\varphi^0VV\\F^0@<s<<F^1@>r>>F^0\end{CD}\]
is commutative.

A graph morphism has the unique path lifting property for $s$ if $\varphi^0, \varphi^1$ are surjective and for every $v\in E^0$ and every $f\in F^1$ with $s(f)=\varphi^0(v)$ there is a unique $e\in E^1$ such that $\varphi^1(e)=f$ and $s(e)=v$.  A graph morphism $\varphi$ is a {\em graph covering} if both $\varphi^0, \varphi^1$ are covering maps. In that case, $\varphi$ has the unique path lifting property with respect to both $s$ and $r$.

An isomorphism of topological graphs is a graph morphism $\varphi=(\varphi^0,\varphi^1)$ such that $\varphi^i$ is
a homeomorphism for $i=0,1$. It follows that $\varphi^{-1}=((\varphi^0)^{-1},(\varphi^1)^{-1})$ is also a graph morphism. We denote by $\aut(E)$ the group of automorphisms of a topological graph $E$.    

A locally compact group $G$ acts on $E$ if there are continuous maps $\lambda^i: G\times E^i\to E^i$, and we write $\lambda^i(g, x)= \lambda^i_g(x)$ for $i=0,1$, such that  $g\mapsto \lambda_g$ is a group homomorphism from $G$ into $\aut(E)$. The action $\lambda$ is called {\em free} if $\lambda^0_g(v)=v$ for some $v\in E^0$ implies $g=1_G$. Note that in this case the action of $G$ is also free on the edges of $E$. 
The action is called {\em proper} if the maps
$G\times E^0\to E^0\times E^0, (g,v)\mapsto (\lambda^0_g(v),v)$ and $G\times E^1\to E^1\times E^1, (g,e)\mapsto (\lambda^1_g(e),e)$ are proper.
In fact,  by Proposition~2.1.14 in \cite{AR}, it is  sufficient to require properness of the first map.
We will frequently write $g\cdot e$ for $\lambda^1_g(e)$, and similarly for the action on vertices.

\begin{defn} Let $P,X$ be locally compact Hausdorff spaces, and let $G$ be a locally compact group. A map $P\stackrel {q}{\longrightarrow} X$ is called a {\em  principal $G$-bundle} if there is a
free and proper action of $G$ on $P$ and $q$ induces an identification
of $P/G$ with $X$.  
\end{defn}

Note that  we make no assumption regarding local triviality. 
Principal bundles are thoroughly discussed in \cite{P} and \cite{husemoller} 
using various definitions.
By \cite[Remark~4.64]{RW} the notion of  principal
$G$-bundle in \cite{husemoller} is equivalent to our notion 
determined by a free and proper action.

Given a  principal $G$-bundle $P\stackrel {q}{\longrightarrow} X$ and a
continuous function $f: Y \to X$, we may view the pull-back $f^*(P) =
\{ (y, p) \in Y\times P  : f(y) = q(p) \}$ as a   principal $G$-bundle
$f^*(P)\stackrel {\pi_1}{\longrightarrow} Y$ where the action of $G$
is given by $g \cdot (y, p)  = (y, g \cdot p)$. 

The following result may be found in \cite[Theorem~4.4.2]{husemoller}
(see also  the discussion preceding  \cite[Proposition~1.3.4]{P}).

\begin{lem}
\label{husemoller}
Suppose we are given  principal $G$-bundles $P_i\stackrel {q_i}{\longrightarrow} X_i$ 
for $i = 1, 2$ and a  pair of continuous maps $f : X_1 \to X_2$, 
$\tilde f : P_1 \to P_2$ such that $\tilde f$ is equivariant and $q_2 
\tilde f = fq_1$. Then $\tht_f : P_1 \to f^*(P_2)$ given by $\tht_f
(p) = (q_1(p), \tilde f(p))$ is an isomorphism of 
$G$-bundles.
\end{lem}

The following is a structure theorem for 
topological graphs equipped with a
free and proper action of
a locally compact group:
the quotient is a topological graph,
and moreover the original graph can be reconstructed
from the quotient graph and
certain additional data.

\begin{thm}
\label{structure}
Let $(F^0, F^1, s, r)$ be a topological graph and $G$ be a locally compact group.  
    Given a  principal $G$-bundle $P\stackrel {q}{\longrightarrow}F^0$ 
    and an isomorphism of pull-backs $\tht : s^*(P) \cong r^*(P)$,
     we may construct a topological graph
     $(E^0, E^1, \tilde s, \tilde r)$ with a free and proper action of $G$ by setting
     $E^0 := P$ and $E^1 :=  s^*(P)$ \(with structure maps defined in the obvious way\).
Conversely, every  topological graph $(E^0, E^1, \tilde s,\tilde r)$ on which $G$ acts freely 
    and properly is isomorphic to one arising in this way.
\end{thm}     
\begin{proof}
First suppose we are given $P$, $q$, and $\theta$.
Define $\tilde s : s^*(P) \to P$ by $\tilde s = \pi_2$, i.e., $\tilde s(e,v)=v$, and 
$\tilde r : s^*(P) \to P$ by $\tilde r = \pi_2 \circ \tht$.  To see that $\tilde s$ is a local homeomorphism, fix a point
in $\tilde e = (e, v)  \in  s^*(P)$. Since $s$ is a local homeomorphism, there is an open set
$U \subset F^1$ containing $e$ such that the restriction $s|_U$ is injective.
It follows that the restriction $\tilde s|_{\pi_1^{-1}(U)}$ is also injective, and hence that
$\tilde s$ is a local homeomorphism.

The following commutative diagram illustrates the above construction:
\[
\xymatrix{
&r^*(P) \ar[dl]_{\pi_2} \ar[ddr]_{\pi_1}
&&E^1=s^*(P) \ar[dr]^{\tilde s=\pi_2} \ar[ddl]^{\pi_1}
\ar[ll]_(.7)\theta^(.7)\cong
\ar[dlll]|!{[ll];[ddl]}\hole^(.3){\tilde r}
\\
E^0=P \ar[d]_q
&&&&E^0=P \ar[d]^q
\\
F^0
&&F^1 \ar[ll]^r \ar[rr]_s
&&F^0
}
\]
Since $\pi_1\circ\theta=\pi_1$,
$\theta$ has the following form:
\[
\theta(e,v)=(e,\tilde r(e,v))\midtext{for}(e,v)\in s^*(P).
\]

The converse essentially follows from Lemma~\ref{husemoller} above;
the only non-obvious bit is showing that the quotient is a topological graph.
So, suppose $G$ acts freely and properly on a topological graph $E=(E^0, E^1, \tilde s,\tilde r)$.
Let $F^i=E^i/G$ be the quotient spaces, with quotient maps $q^i:E^i\to F^i$, for $i=0,1$. Since the $G$-action on $E$ is compatible with the range and source maps, we have the following diagram of  principal $G$-bundles:
\[
\xymatrix{
E^0 \ar[d]_{q^0}
&E^1 \ar[l]_{\tilde r} \ar[r]^{\tilde s} \ar[d]_{q^1}
&E^0 \ar[d]^{q^0}
\\
F^0
&F^1 \ar[l]^r \ar[r]_s
&F^0,
}
\]
where 
$r,s:F^1\to F^0$ are continuous, and $s$ is open.
We must show that $F=(F^0,F^1,s,r)$ is a topological graph, and that there is an isomorphism of pull-backs $\theta:s^*(E^0)\to r^*(E^0)$ such that $E$ is isomorphic to the topological graph with the same vertex space $E^0$, but with edge space $s^*(E^0)$, source map $\pi_2$, and range map $\pi_2\circ\theta$.
By Lemma~\ref{husemoller} there are $G$-bundle isomorphisms $\theta_r,\theta_s$ making the following diagram commute (without the top arrow $\theta$):
\[
\xymatrix{
&r^*(E^0) \ar[dl]_{\pi_2}
\ar@/_1pc/[ddr]|!{[dr];[dl]}\hole_(.6){\pi_1}
&&s^*(E^0) \ar[dr]^{\pi_2} \ar@{-->}[ll]_\theta^\cong
\ar@/^1pc/[ddl]|!{[dl];[dr]}\hole^(.6){\pi_1}
\\
E^0 \ar[d]_{q^0}
&&E^1 \ar[ll]_(.7){\tilde r} \ar[rr]^(.7){\tilde s} \ar[d]_(.4){q^1}
\ar[ul]_(.4){\theta_r}^(.4)\cong
\ar[ur]^(.4){\theta_s}_(.4)\cong
&&E^0 \ar[d]^{q^0}
\\
F^0
&&F^1 \ar[ll]^r \ar[rr]_s
&&F^0.
}
\]
Of course we define the $G$-bundle isomorphism $\theta$ so that the top triangle (and hence the entire diagram) commutes.

It only remains to show that $s:F^1\to F^0$ is locally injective.
Let $e\in F^1$, and put $v=s(e)$.
Fix $u\in E^0$ with $q^0(u)=v$,
so that $(e,u)\in s^*(E^0)$.
Since 
$E$ is a topological graph,
we can find open neighborhoods $U$ of $e$ in $F^1$ and $V$ of $u$ in $E^0$ such that $\pi_2:s^*(E^0)\to E^0$ is injective on
\[
(U\times V)\cap s^*(E^0).
\]
Since $q^0:E^0\to F^0$ is open, we can shrink $U$ if necessary so that $s(U)\subset q^0(V)$.

Let $e_1,e_2\in U$, and assume that
\[
s(e_1)=s(e_2)=v'.
\]
Choose $u'\in V$ such that $q^0(u')=v'$.
Then $(e_1,u'),(e_2,u')\in (U\times V)\cap s^*(E^0)$, so we must have $e_1=e_2$.
\end{proof}

\example If we are given a finitely generated locally compact group $G$ with generating set $S$,  then $G$ acts freely and properly on its Cayley graph $E=E(G,S)$  by $\lambda^0_g(g')=gg'$ and  $\lambda^1_g (h,g')=(h, gg')$. The quotient graph $E/G$ has $|S|$ loops and one vertex.

\example If $E=(E^0,E^1,s,r)$ is a  topological graph, $G$ is a locally compact group, and $c:E^1\to G$ is a continuous function, then   $\varphi: E\times_cG\to E,\; \varphi(g,x)=x$ is a   principal $G$-bundle map. 
 Indeed, $G$ acts freely and properly on $E\times_cG$ by $\lambda^0_g (v,h)=(v, gh)$ and
$\lambda^1_g (e,h)=(e, gh)$. The source and range maps are equivariant, and the quotient graph is isomorphic to $E$. 

The following result characterizes skew products:


\begin{cor}
\label{trivial}
The topological graph constructed in Theorem~\ref{structure} from an isomorphism  $\tht : s^*(P) \cong r^*(P)$ is
$G$-equivariantly
isomorphic to a skew product $F\times_c G$ if and only if 
the  $G$-bundle $P$ over $F^0$
is trivial.
\end{cor}
\begin{proof}
The forward direction is trivial, so assume that the $G$-bundle $q:P\to F^0$ is trivial, and let $E$ be the topological graph constructed from an isomorphism $\theta:s^*(P)\to r^*(P)$.
The pull-back of a trivial bundle is trivial,
so most of what we have to do is quite straightforward;
the only slightly non-obvious bit is that the isomorphism $s^*(P)\cong F^1\times G$ of $G$-bundles can be promoted to a topological-graph isomorphism $E\cong F\times_c G$ for a suitable cocycle $c$.

Let
$\varphi^0:P\to F^0\times G$
be a $G$-bundle isomorphism.
We must show that there are:
\begin{itemize}
\item a $G$-bundle isomorphism $\varphi^1:E^1\to F^1\times G$,
and 

\item a cocycle
$c:F^1\to G$
\end{itemize}
such that $\varphi=(\varphi^0,\varphi^1):E\to F\times_c G$ is an isomorphism of topological graphs.

The diagram
\begin{equation}
\label{big}
\xymatrix{
P \ar[dd]_q \ar[dr]^{\varphi^0}_\cong
&&&E^1 \ar[dd]^{\pi_1}
\ar[dl]_{\varphi^1}^\cong \ar[dr]^{\varphi^1}_\cong
\ar[lll]_{\tilde r} \ar[rrr]^{\tilde s=\pi_2}
&&&P \ar[dd]^q \ar[dl]_{\varphi^0}^\cong
\\
&F^0\times G \ar[dl]^{\pi_1}
&F^1\times G \ar[dr]_{\pi_1} \ar[l]_{r'}
&&F^1\times G \ar[dl]^{\pi_1} \ar[r]^{s\times\id}
&F^0\times G \ar[dr]_{\pi_1}
\\
F^0
&&&F^1 \ar[lll]^r \ar[rrr]_s
&&&F^0,
}
\end{equation}
where
\[
r':(F\times_c G)^1=F^1\times G\to F^0\times G=(F\times_c G)^0
\]
is the range map for the skew product, given by
\[
r'(e,g)=(r(e),gc(e)),
\]
illustrates the hypothesis and conclusion.
Keep in mind that $E^0=P$ and $E^1=s^*(P)$.
We are given the map $\varphi^0$, and since it is a $G$-bundle isomorphism the extreme left and right triangles commute.

It is an elementary fact about pull-backs that there is a $G$-bundle isomorphism $\varphi^1:s^*(P)\to F^1\times G$ making the right half of the diagram \eqref{big} commute.
We can now \emph{define} $r':F^1\times G\to F^0\times G$ so that the left half of \eqref{big} commutes.
It remains to show that there is a continuous map $c:F^1\to G$ such that $r'$ has the form
\[
r'(e,g)=(r(e),gc(e)).
\]
Since
\[
\pi_1\circ r'=r\circ\pi_1:F^1\times G\to F^0,
\]
we have
$\pi_1(r'(e,1))=r(e)$.
Define $c:F^1\to G$ by
\[
c(e)=\pi_2(r'(e,1)).
\]
Then $c$ is continuous, and we have
\begin{align*}
r'(e,g)
&=r'(g\cdot(e,1))
\\&=g\cdot r'(e,1)
\\&=g\cdot (r(e),c(e))
\\&=(r(e),gc(e)).
\qedhere
\end{align*}
\end{proof}

\begin{lem}
\label{compact intersect}
If we are given  principal $G$-bundles $P_i\stackrel {q_i}{\longrightarrow} X_i$  for $i = 1, 2$ and a  pair of 
continuous maps $f : X_1 \to X_2$,  $\tilde f : P_1 \to P_2$ such that $\tilde f$
is equivariant and $q_2  \tilde f = fq_1$, and compact sets $K \subset X_1$ and $L \subset P_2$,
then $(q_1)^{-1}(K) \cap \tilde f ^{-1}(L)$ is a compact subset of $P_1$.
\end{lem}
\begin{proof}
Observe that $f^*(P_2)$ is a closed subset of $X_1 \times P_2$ and hence
$(K\times L) \cap f^*(P_2)$ is compact.   Therefore,
\[
q_1^{-1}(K) \cap \tilde f ^{-1}(L) = (\tht_f )^{-1}((K\times L) \cap f^*(P_2))
\]
is also compact.
\end{proof}

\begin{cor}
\label{compactly trivial}
Let a locally compact group $G$ act freely and properly on a topological graph $E=(E^0,E^1,s,r)$,
let $q_1:E^1\to E^1/G$ be the quotient map, and
let $K\subset E^1/G$ be compact. Then:
\begin{enumerate}
\item For every compact subset $L\subset E^0$, 
both intersections
\[
q_1\inv(K)\cap r\inv(L)\midtext{and}
q_1\inv(K)\cap s\inv(L)
\]
are compact.

\item There exists $d\ge 0$ such that
\[
\bigl|q\inv(K)\cap s\inv(v)\bigr|\le d
\midtext{for all}v\in E^0.
\]
\end{enumerate}
\end{cor}

\begin{proof}
(1) follows immediately from Lemma~\ref{compact intersect}.
For (2), note first of all that the conclusion is unaffected by replacing $E$ with an isomorphic topological graph,
so by Theorem~\ref{structure} we may assume that $E$ is constructed from
a topological graph $F$, a  principal $G$-bundle $q:P\to F^0$, and
an isomorphism $\theta:s^*(P)\to r^*(P)$.
Thus $E^1=s^*(P)$, with 
quotient maps $\pi_1:s^*(P)\to F^1$ on edges and
$q:P\to F^0$ on vertices,
and source map $\pi_2:s^*(P)\to P$.
Note that for $v\in P$, if we put $u=q(v)\in F^0$ then we have
\[
\pi_1\inv(K)\cap \tilde s\inv(v)
=\{(e,v)\in s^*(P):e\in K\}
=\bigl(K\cap s\inv(u)\bigr)\times \{v\}.
\]
Thus
\[
\bigl|\pi_1\inv(K)\cap \tilde s\inv(v)\bigr|
=\bigl|K\cap s\inv(u)\bigr|.
\]
Since $s:F^1\to F^0$ is a local homeomorphism and $K$ is compact,
the cardinalities of the intersections $K\cap s\inv(u)$ for $u\in F^0$ are bounded above by some fixed real number.
\end{proof}

\remark If $G$ is discrete and $G$ acts freely and properly on $E$, then the  morphism $q:E\to E/G$ is  a  graph covering, since $G$ acts properly discontinuously (see 4.69 in \cite{RW}).

\begin{remark} If $E$ and $G$ are discrete and $G$ acts freely on $E$, then by Theorem~2.2.2 in \cite{GT}, there is a cocycle $c:(E/G)^1\to G$ such that $(E/G)\times_cG\cong E$ in an equivariant way. This result can be obtained from Corollary \ref{trivial}, since any  principal $G$-bundle over $E^0$ is trivial. In general, not every free and proper action on a topological graph is associated to a skew product.
\end{remark}

Recall that by \cite[Proposition 8.9]{Ka3} a topological graph $E$ is minimal iff every 
orbit space $\Orb(v,e)$ (see \cite[Definition 4.9]{Ka3}) is dense in $E^0$, where 
$v \in E^0$ and $e$ is a negative orbit of $v$.
If $E^0=E^0_{\text{rg}}$, a negative orbit of $v$ is an infinite path 
$e=e_1e_2\cdots\in E^{\infty}$ ending at $v$ and
\[
\Orb(v,e)=\{r(e') \mid e' \in E^m,\  s(e')=s(e_n) \textrm{  for some }n \}.
\]

\begin{theorem} Let $E$ be a minimal topological graph and let $q:E\to F$ be a graph covering.
Suppose that $E^0$ is not discrete and $E^0=E^0_{\text{rg}}$. Then $F$ is also minimal, 
$F^0$ is not discrete and $F^0=F^0_{\text{rg}}$.  In particular, both $C^*(E)$ and $C^*(F)$ are simple.
\end{theorem}
\begin{proof} Indeed, the unique path lifting property ensures that every orbit space in $F$ is dense. The second part follows from Corollary 8.13 in \cite {Ka3}.
\end{proof}
\begin{corollary}
 Let $E$ be a minimal topological graph and suppose that $E^0$ is not discrete and 
 $E^0=E^0_{\text{rg}}$.   If the discrete group $G$ acts freely and properly on $E$,  
 then $C^*(E/G)$ is simple.
\end{corollary}
\example
Now let $E=E(G,S)$ be as in Example \ref{cayley} where $G= {\R}$ and 
$S=\{1,\theta\}$, with $\theta$ is irrational.  Then  $E^0={\R}$,  
$E^1=\{1,\theta\}\times{\R}$.  Translation by $\Z$ gives a free and proper 
action on both $E^0$ and $E^1$ which intertwines the structure maps $r, s$.
Hence, translation induces a free and proper action of $\Z$ on $E$.   Since $E$ is 
a minimal topological graph, $E^0$ is not discrete and $E^0=E^0_{\text{rg}}$, we have
by the above corollary that $C^*(E/\Z)$ is simple.

\section{The fundamental group of  a topological graph}

In this section, we define the fundamental group of a topological graph $E=(E^0, E^1, s,r)$ and 
the notion of a universal covering of $E$, using a single topological space $R(E)$ called the 
geometric realization. We need to assume that the geometric realization has a universal covering. 
In this section we need not  assume that the source map is a local homeomorphism, even 
though in all of our examples this is the case.

 For each $e\in E^1$ we formally denote the reversed edge by $\bar{e}$, where $s(\bar{e})=r(e)$ and $r(\bar{e})=s(e)$. The set of reversed edges is denoted $\bar{E}^1$. A {\em walk} in $E$ is a sequence $w=e_1\cdots e_n$ where $e_i\in E^1\cup \bar{E}^1$ such that $s(e_i)=r(e_{i+1})$ for $i=1,...,n-1$. We define $s(w):=s(e_n)$ and $r(w):=r(e_1)$. A vertex will be considered as a trivial walk. A walk $w$ is reduced if it does not contain the subword $e\bar{e}$ for any $e\in E^1\cup\bar{E}^1$. Denote by $E^{rw}$ the space of reduced walks, with product topology. Note that $E^0$ is viewed as a subset of $E^{rw}$. For $w=e_1\cdots e_n\in E^{rw}$, the reverse walk $\bar{w}$ is $\bar{e}_n\cdots\bar{e}_1$.

The following definition is modeled on  \cite[Definition~III.$\mathcal{G}$.3.1]{BH}.

\definition 
An $E$-path $c=(w_0, c_1, w_1,...,c_k,w_k)$  over a partition $0=t_0\le...\le t_k=1$ of the interval $[0,1]$ consists of continuous maps $c_i:[t_{i-1},t_i]\to E^0$ and reduced walks
$w_i\in E^{rw}$ such that $r(w_i)=c_{i+1}(t_i)$ for $i=0,1,...,k-1$ and $s(w_i)=c_i(t_i)$ for $i=1,2,...,k$. The initial point of $c$ is $x=s(w_0)$; its terminal point is $y=r(w_k)$. We say that $c$ joins $x$ to $y$. If $w_0$ and $w_k$ are trivial walks (vertices), they can be dropped in the notation for $c$. In particular, a map $c:[0,1]\to E^0$ can be considered as an $E$-path.

\definition A topological graph $E$ is said to be {\em connected} if given any two vertices $x,y\in E^0$, there is an $E$-path $c$ joining $x$ and $y$.

\definition \label{def:geo}
The {\em geometric realization} of a  topological graph $E$ is the topological space $R(E)$ obtained from the disjoint union $E^1\times [0,1]\sqcup E^0$ by identifying $(e,0)$ with $s(e)$ and $(e,1)$ with $r(e)$ (a kind of double mapping torus). We will identify $E^0$ with a subspace of $R(E)$ in the obvious way. Also, we embed $E^1$ in $R(E)$ by $e \in E^1 \mapsto (e, \frac{1}{2})$.

\remark 
Notice that if $E$ is connected, then $R(E)$ is path connected.  Let $\varphi : E \to F$ be a graph morphism. Then $\varphi$ yields a
natural map 
\[
E^1\times [0,1]\sqcup E^0 \to F^1\times [0,1]\sqcup F^0
\]
which then induces a map $R(\varphi) : R(E) \to R(F)$.
Moreover, for $i =0,1$ the following diagram  commutes:
\[
\begin{CD}E^i@>>>R(E)\\@V\varphi^iVV@VR(\varphi)VV\\F^i@>>>R(F)\end{CD}
\]
where  the horizontal maps are the canonical embeddings given in Definition~\ref{def:geo}.
Observe that $R(\varphi)$ is continuous; moreover, $R(\varphi)$ is a covering if $\varphi$ is.

\remark If the group $G$ acts on the topological graph $E$, then $G$ acts on $R(E)$ by $g\cdot (e,t)=(\lambda^1_g(e),t)$ for $e\in E^1, t\in [0,1]$, and by $g\cdot v=\lambda^0_g(v)$ for $v\in E^0$. Since $\lambda^i_g$ commute with $s$ and $r$ for $i=0,1$, the action is well defined.

\definition 
Let $E=(E^0, E^1, s,r)$ be a a connected topological graph.  We define the fundamental group
of $E$ by  $\pi_1(E) :=  \pi_1(R(E))$.  We say that  $E$ is  simply connected if 
$\pi_1(E)$ is trivial (i.e. $R(E)$ is simply connected).  
We say that a graph morphism  $p: \tilde E \to E$  is a universal covering if
$p$ is a covering and $\tilde E$ is connected and simply connected
(or briefly that $\tilde E$ is a universal cover). 

\remark The fundamental group $\pi_1(E)$ acts freely on $\tilde E$, and the orbit space is isomorphic to $E$. Any subgroup $H$ of $\pi_1(E)$ will determine an intermediate covering of $E$, by taking the graph $\tilde E/H$. Recall that the fundamental group of a finite graph $E$ is the free group with
$|E^1|-|E^0|+1$ generators, see \cite[Lemma 4.10]{KP}.

\begin{proposition}\label{univ}
Let $E$ be a connected  topological graph 
and suppose that $R=R(E)$ is locally path-connected and semi-locally simply connected.
Then $R$ has a universal covering space  $\tilde R$.  Moreover, $\tilde R \cong R(\tilde E)$
where $\tilde E$ is a universal cover of  $E$.
\end{proposition}
\begin{proof}
By \cite[Theorem~10.2]{massey}  $R$ has a universal covering space  $\tilde R$.
Let $\pi: \tilde R\to R$ be the canonical map,  let $\tilde E^0=\pi^{-1}(E^0)$, and let $\tilde E^1=\pi^{-1}(E^1\times\{1/2\})$. In order to define the range and the source maps, we  use the unique path lifting property of the map $\pi$. Let $\tilde e\in\tilde E^1$, then $\pi(\tilde e)=(e, 1/2)$ for some $e \in E^1$. Join $(e, 1/2)$ with the image of $(e,0)$ in $R$ by the path $t \mapsto (e, \frac{1}{2}(1-t)), t\in [0,1]$. Lift this path to a path in $\tilde R$, and define $s(\tilde e)$ to be the endpoint of the lifted path, which belongs to $\tilde E^0$. The range map is obtained similarly by joining $(e,1/2)$ with the image of $(e,1)$.
\end{proof}

\begin{corollary}\label{suffcond}
Consider a connected topological graph $E$. If both $r, s : E^1 \to E^0$ are finite-to-one covering maps and
$E^0$ is  locally contractible (each point has a local base of contractible neighborhoods), then $R(E)$ is also locally contractible. In particular,  $E$ has a universal cover. 
\end{corollary}
\begin{proof}
It suffices to show that $R(E)$ is locally contractible since the hypotheses of the above proposition will be satisfied.
A point $(e,t)$ in the image of $E^1\times (0,1)$ has a local base of contractible neighborhoods of the form $V\times W$, where $V$ is a contractible neighborhood of $e$ in $E^1$ and $W$ is a contractible neighborhood of $t$ in $(0,1)$. For the image of $v\in E^0$ in $R(E)$, take $V\subset E^0$ a contractible neighborhood of $v$ which is evenly covered by the maps $s$ and $r$. Then  we claim that the image $N$ of $s^{-1}(V)\times [0,\varepsilon)\cup r^{-1}(V)\times (1-\varepsilon, 1]$ in $R(E)$ is a contractible neighborhood of the image of $v$ in $R(E)$ for any $0<\varepsilon<1/2$. Indeed, $s^{-1}(V)$ is a disjoint union $V_1\cup...\cup V_p$ , and $r^{-1}(V)$ is a disjoint union $U_1\cup...\cup U_q$ of homeomorphic copies of $V$. Let $S$ be the star with $p+q$ branches obtained by gluing $p+q$ copies of $[0,\varepsilon)$ at $0$. Then $N$ is homeomorphic to $S\times V$, which is contractible.
\end{proof}

It would be nice to have a weaker condition on $E$ that would ensure the existence of the universal cover.

\remark In order to obtain other coverings of a graph $E$, we may consider a subgroup $H$ of the fundamental group $\pi_1(E)$ which will act on $\tilde R$, and take the corresponding topological graph of the quotient space $\tilde R/H$, constructed as in the proof of Proposition~\ref{univ}.

\remark The fundamental group of $(E^0, E^1, s,r)$ is isomorphic to the fundamental group of the opposite graph $(E^0, E^1, r, s)$, obtained by interchanging the maps $s$ and $r$. The natural embeddings described in~\ref{def:geo} induce maps $\pi_1(E^0, v)\to \pi_1(E)$ and $\pi_1(E^1,e)\to \pi_1(E)$ for fixed $v\in E^0$ and $e\in E^1$.

\example Consider the topological graph $E$ with $E^0=E^1={\mathbb T}$ and source and range maps $s(z)=z, r(z)=e^{2\pi i\theta}z$ for $\theta$ irrational. The geometric realization is homeomorphic to the 2-torus ${\mathbb T}^2$, hence the fundamental group $\pi_1(E)$ is isomorphic to ${\mathbb Z}^2$. The universal covering graph is $\tilde E=(\tilde E^0, \tilde E^1, s, r)$, where $\tilde E^0=\tilde E^1={\mathbb R}\times{\mathbb Z}$, with $s(y,k)=(y,k), r(y, k)=(y+\theta, k+1)$. The action of ${\mathbb Z}^2$  on $\tilde E$ is given  by $(j, m)\cdot(y,k)=(j+y+m\theta, k+m)$, and $\tilde E/{\mathbb Z}^2\cong E$.  Any other connected covering of $E$ is of the form $\tilde E/H$, where $H$ is a subgroup of ${\mathbb Z}^2$.

More generally, let $X$ be a 
compact space which admits a universal covering space $\tilde X$, and let $h:X\to X$ be a homeomorphism. The geometric realization of the corresponding topological graph $E$ with $E^0=E^1=X$, $s=id$ and $r=h$ is homeomorphic to the mapping torus of $h$, obtained from $X\times [0,1]$ by identifying $(x,1)$ with $(h(x),0)$. The universal covering graph $\tilde E$ has $\tilde E^0= \tilde E^1=\tilde X\times {\mathbb Z}$. The source and range maps are given by $s(y,k)=(y,k), r(y,k)=(\tilde h(y), k+1)$, where $\tilde h:\tilde X\to \tilde X$ is a lifting of $h$, $y\in \tilde X$ and $k\in {\mathbb Z}$. The map $h$ induces an automorphism $h_*$ of $\pi_1(X)$, and the fundamental group of $E$ is isomorphic to the semi-direct product $\pi_1(X)\rtimes {\mathbb Z}$ defined using $h_*$. The action of $\pi_1(X)\rtimes {\mathbb Z}$ on $\tilde X\times {\mathbb Z}$ is given by $(g,m)\cdot (y,k)=(g\cdot\tilde h^m(y), k+m)$. 


\example Let again $E^0=E^1={\mathbb T}$ with source and range maps $s(z)=z^n, r(z)=z^m$ for $n,m$ positive integers. The  geometric realization $R(E)$ is obtained from the cylinder $E^1\times [0,1]$, where the two boundary circles are identified using the source and range maps. Alternatively, we may start with a rectangle, such that  the left and right edges are labeled by $a$. The top edge $E^1\times\{0\}$ is divided into $n$ segments called $b$, and the bottom edge $E^1\times\{1\}$ is divided into $m$ segments also called $b$. By making the identifications, we get the geometric realization $R(E)$. Note that here the arrows are for identification purposes. \begin{figure}[htbp] 
   \centering
   \includegraphics[trim= 0.2in 0in 0.2in 1.2in, clip, width=3in]{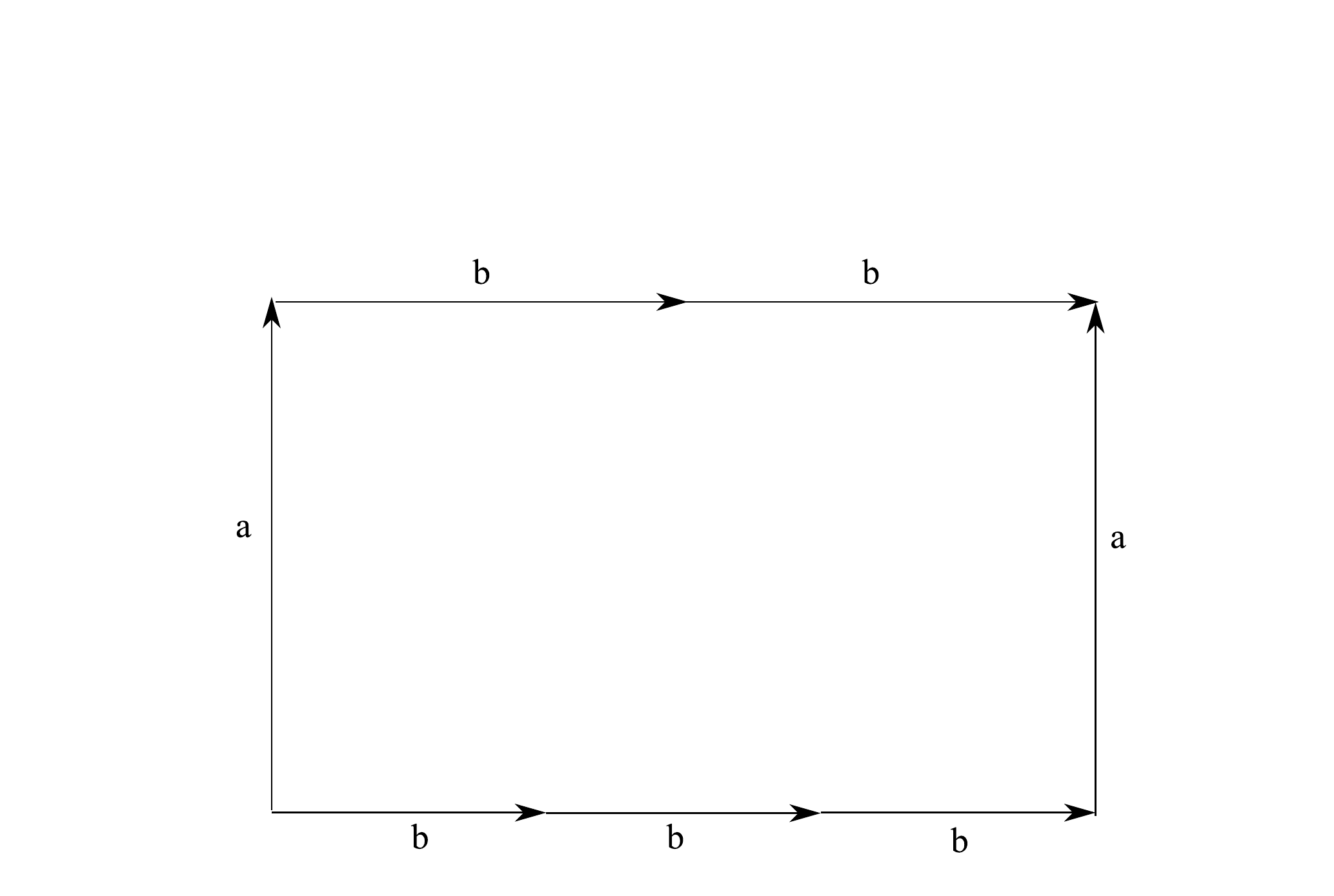} 
   \caption{The case $n=2, m=3.$}
\end{figure}

The segments $a,b$ become generators in the fundamental group
 $\pi_1(R(E))$, which  is isomorphic to the Baumslag-Solitar group 
\[B(n,m)=\langle a,b \mid ab^na^{-1}=b^m\rangle.\] 
For $n=1$ or $m=1$, this group is a semi-direct product and it is amenable. For $n\neq 1, m\neq 1$ and $n,m$ relatively prime, it is not amenable (see Example~E.11 in \cite{BO}).

The universal covering space of ${R(E)}$ is the Cayley complex $\tilde{R}$ of $B(n,m)$, obtained from the Cayley graph  by filling out the rectangles (see, for example page~122 in \cite{LS}). It is the cartesian product $T\times {\mathbb R}$, where $T$ is the Bass-Serre tree of $B(n,m)$, viewed as an HNN-extension. Recall that, given a group $G$, a subgroup $H\subset G$, and a monomorphism $\tau: H\to G$, then the HNN extension $G*_H\tau$ is generated by $G$ and a letter $a$ such that $aha^{-1}=\tau(h)$ for $h\in H$ (see for example \cite{Bau}). In our case, $G=\langle b\rangle\cong {\mathbb Z},\; H=\langle b^n\rangle\cong n{\mathbb Z}$, and $\tau(b^n)=b^m$. 

The free action of the group $B(n,m)$  on $\tilde{R}$  is the extension  to the $2$-cells of the usual left action on the Cayley graph. 
A piece of the Cayley complex of $B(2,3)$ together with the action of the generators is illustrated in Figure 2. 
\begin{figure}[htbp] 
   \centering
   \includegraphics[trim= 0in 0.5in 0in 0in, clip, width=5in]{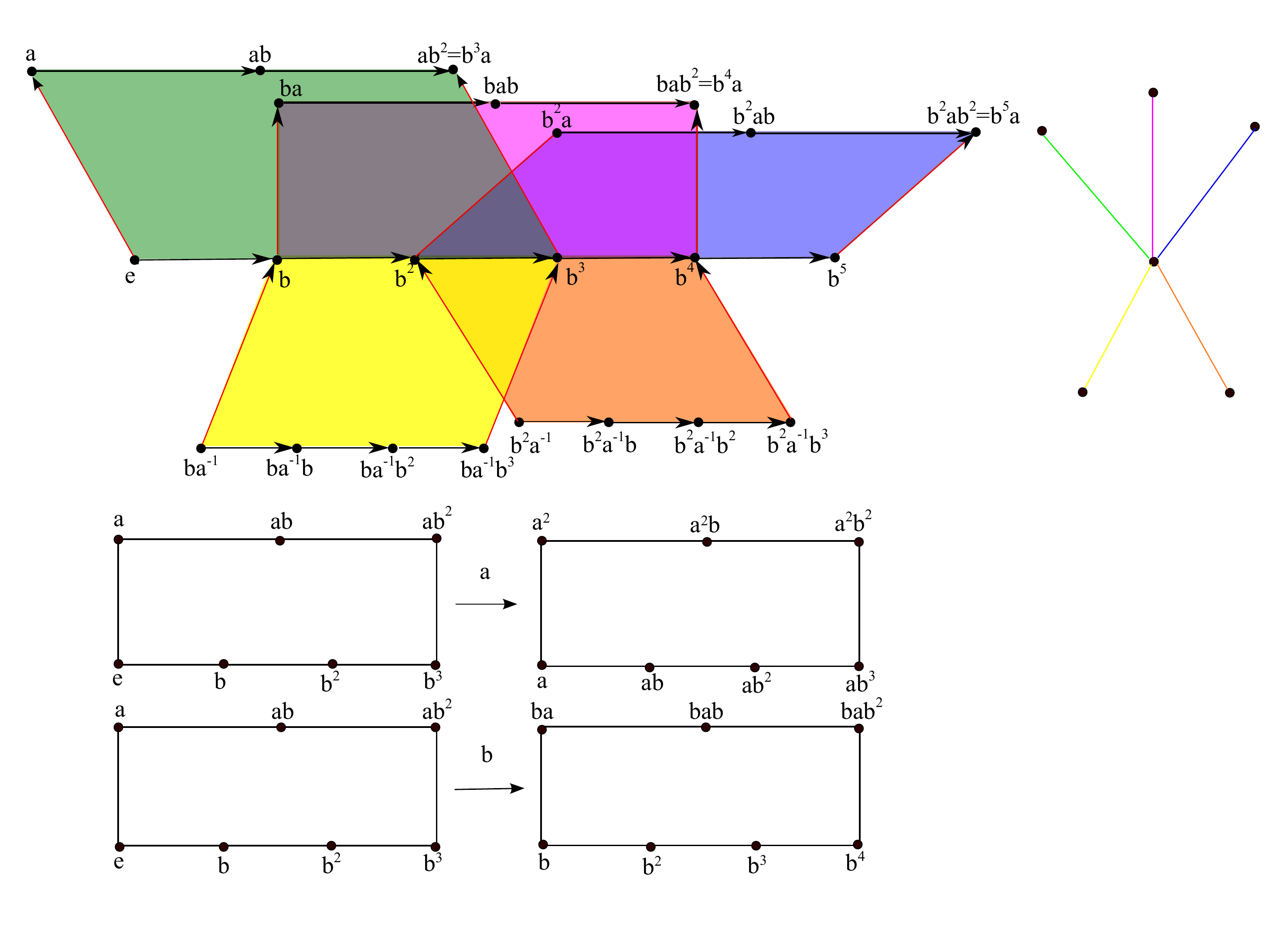} 
   \caption{Cayley complex for $B(2,3).$}
\end{figure}

The $1$-skeleton is the directed Cayley graph of $B(2,3)$, where the generators $a,b$ multiply on the right. The group action is given by left multiplication. In the corresponding tree $T$, each vertex has $5$ edges. The vertex set $T^0$ is identified with the left cosets $g\langle b\rangle\in B(2,3)/\langle b\rangle$, and the edge set $T^1$ with the left cosets $g\langle b^2\rangle\in B(2,3)/\langle b^2\rangle$. The source and range maps are given by $s(g\langle b^2\rangle)= g\langle b\rangle, r(g\langle b^2\rangle)=ga^{-1}\langle b\rangle$ for $g\in B(2,3)$. 

Using Proposition~\ref{univ}, we can describe the universal covering graph $\tilde{E}$ of $E$. We have $\tilde{E}^0\cong T^0\times {\mathbb R}$, $\tilde{E}^1\cong T^1\times {\mathbb R}$ with $\tilde{s}(t,y)=(s(t), ny), \tilde{r}(t,y)=(r(t),my)$ for $t\in T^1$ and $y\in{\mathbb R}$. The group $B(n,m)$ acts freely and properly on $\tilde{E}$, and the quotient graph $\tilde{E}/B(n,m)$ is isomorphic to $E$. In particular, note that the  topological graph $\tilde{E}$ is not a skew product. In the next section, we show that $C^*(\tilde{E})\rtimes_r B(n,m)$ is strongly Morita equivalent to $C^*(E)$.



\section{Group actions on  $C^*$-correspondences}
\label{actions}

\definition Let $G$ be a locally compact group, and let ${\mathcal H}$ be  a $C^*$-correspondence over $A$. We say that $G$ acts on ${\mathcal H}$ if there is a map \[G\times {\mathcal H}\to {\mathcal H},\;\; (g,\xi)\mapsto g\cdot\xi\] such that $g\mapsto g\cdot \xi$ is continuous and $\xi\mapsto g\cdot\xi$ is linear, and a continuous  action of $G$ on $A$ by $*$-automorphisms such that  \[ \langle g\cdot\xi,g\cdot\eta\rangle=g\cdot\langle\xi,\eta\rangle ,\; g\cdot(\xi a)=(g\cdot\xi)(g\cdot a),\; g\cdot(\varphi(a)\xi)=\varphi(g\cdot a)(g\cdot\xi),\]
where $\varphi:A\to{\mathcal L}({\mathcal H})$ defines the left multiplication.

\begin{remark}
There is an action of $G$ on ${\mathcal L}({\mathcal H})$ given by \[(g\cdot T)(\xi)=g\cdot(T(g^{-1}\cdot\xi)),\] such that $g\mapsto g\cdot T$ is continuous with respect to the strict topology on ${\mathcal L}({\mathcal H})$.
We have $g\cdot(\theta_{\xi, \eta})=\theta_{g\cdot\xi,g\cdot\eta}$, which gives an action on ${\mathcal K}({\mathcal H})$. Indeed, 
\[g\cdot(\theta_{\xi, \eta})(\zeta)=g\cdot(\xi\langle\eta,g^{-1}\cdot\zeta\rangle)=(g\cdot\xi)\langle g\cdot\eta,\zeta\rangle=(\theta_{g\cdot\xi,g\cdot\eta})(\zeta).\]

\end{remark}


 An action of $G$ on the $C^*$-correspondence ${\mathcal H}$ defines an action on the Toeplitz algebra ${\mathcal T}_{\mathcal H}$ by $g\cdot T_\xi=T_{g\cdot\xi}$, and on the Cuntz-Pimsner algebra ${\mathcal O}_{\mathcal H}$, since all defining relations are equivariant. 

The full crossed product ${\mathcal H}\rtimes G$ can be defined 
 by ${\mathcal H}\rtimes G={\mathcal H}\otimes_{\varphi}(A\rtimes G)$, where $\varphi: A\to {\mathcal L}(A\rtimes G)$ is the embedding of $A$ in the multiplier algebra of the crossed product $A\rtimes G$, regarded as a Hilbert module over itself. 
  
 Note that $G$ acts on ${\mathcal H}\rtimes G$ by $g\cdot(\xi\otimes f)=g\cdot\xi\otimes(u_gfu_g^{-1})$, where $\xi\in {\mathcal H}, f\in A\rtimes G$, and $g\cdot a=u_gau_g^{-1}$. Here $u:G\to UM(A\rtimes G)$ is the canonical map into the unitary multipliers. The crossed product ${\mathcal H}\rtimes G$ becomes a $C^*$-correspondence over $A\rtimes G$ in a natural way.
 
Similarly, the reduced crossed product ${\mathcal H}\rtimes_r G$ can be defined as a  $C^*$-correspondence over $A\rtimes_rG$; moreover, there is a natural action of $G$ on ${\mathcal H}\rtimes_r G$ defined as above.

\begin{remark}
This  crossed product ${\mathcal H}\rtimes G$ is isomorphic to the one defined using a completion of $C_c(G, {\mathcal H})$, as in \cite{K} and \cite{HN}. The isomorphism is induced by the
function
\[{\mathcal H}\otimes C_c(G,A)\to C_c(G,{\mathcal H})
\]
which sends $\xi\otimes f$ to the map $g \mapsto  \xi f(g)$.
It is proved in \cite[Theorem~2.10]{HN}  that if $G$ is amenable, then
\[{\mathcal O}_{{\mathcal H}\rtimes G}\cong {\mathcal O}_{\mathcal H}\rtimes G.\]
Note that we also consider actions and crossed products by non-amenable groups.
\end{remark}

\begin{proposition} If $G$ acts on the topological graph $E=(E^0,E^1,s,r)$, then $G$ acts on the $C^*$-correspondence ${\mathcal H}={\mathcal H}(E)$ and hence on $C^*(E)$.
\end{proposition}
\begin{proof} Define $g\cdot\xi(e)=\xi(g^{-1}e)$ for $\xi\in C_c(E^1)$, and $g\cdot f(v)=f(g^{-1}v)$ for
$f\in C_0(E^0)$. Then this action is compatible with the bimodule structure since $s$ and $r$ are equivariant.
\end{proof}

\definition Recall from \cite{Ri1} that the action $\alpha$ of a locally compact group $G$ on a $C^*$-algebra $A$ is {\em proper} if there is a dense $\alpha$-invariant $*$-subalgebra $A_0$ of $A$ such that for every $a,b\in A_0$ the functions 
\[x\mapsto a\alpha_x(b)\;\; \text{and}\;\; x\mapsto \Delta(x)^{-1/2}a\alpha_x(b)\]
are integrable and there exists a (right) inner product $\langle \cdot, \cdot\rangle_r$ with values in
the subalgebra of $M(A)$, which Rieffel denotes by $M(A_0)$, comprising those elements that multiply $A_0$ into itself,
such that
\[c\langle a, b\rangle_r=\int_Gc\alpha_x(a^*b)dx\;\; \text{for all}\;\; c\in A_0.\]
For such an action, 
\[A^\alpha:=\overline{\text{span}}\{\langle a, b\rangle_r: a,b\in A_0\}\subset M(A)\]
is called the {\em generalized fixed-point algebra}. Define a (left) inner product on $A_0$ with values in $A\rtimes_{\alpha, r}G$ by
\[{}_\ell\langle a, b\rangle(x)=\Delta(x)^{-1/2}a\alpha_x(b^*).\] 
The set 
\[I:=\overline{\text{span}}\{{}_\ell\langle a, b\rangle: a,b\in A_0\}\]
is an ideal in $A\rtimes_{\alpha, r}G$, and the closure ${\mathcal Z}$ of $A_0$ in the norm $\|a\|^2:=\|\langle a, a\rangle_r\|$ is an $I-A^\alpha$ imprimitivity bimodule. The action is called {\em saturated} if $I=A\rtimes_{\alpha, r}G$.

We review now the averaging process of \cite[Section~2]{KQR}, which was built upon work in
\cite{OP1, OP2, Q:landstad, twlan}.
Let a locally compact group $G$ act freely and properly on a locally compact Hausdorff space $T$,
and let $\gamma$ be the associated action of $G$ on $C_0(T)$:
\[
\bigl(\gamma_g(f)\bigr)(t)=f(g\inv\cdot t)\midtext{for}g\in G,f\in C_0(T),t\in T.
\]
Also let $\alpha$ be an action of $G$ on a $C^*$-algebra $A$, and suppose that we have an equivariant nondegenerate\footnote{recall that this means that $\varphi(C_0(T))A=A$} homomorphism $\varphi:C_0(T)\to M(A)$.
For $g\in G$ let $\bar{\alpha_g}$ denote the canonical extension of $\alpha_g$ to $M(A)$, and write $M(A)^\alpha=\{a\in M(A):\bar{\alpha_g}(a)=a\text{ for all }g\in G\}$ (not to be confused with Rieffel's generalized fixed-point algebras discussed above).
It was shown in \cite{KQR} that
\[
A_0:=\spn \varphi(C_c(T))A\varphi(C_c(T))
\]
is a dense $*$-subalgebra of $A$, and that there is a linear map $\Phi:A_0\to M(A)^\alpha$
such that
\[
\omega(\Phi(a))=\int_G \omega(\alpha_g(a))\,dg\midtext{for all}a\in A_0,\ \omega\in A^*.
\]
We write $\Phi^\alpha$ when confusion is possible.
\cite{KQR} also shows that $\Phi(A_0)$ is a $*$-subalgebra of $M(A)$, and so its norm closure, denoted by $\fix(A,\alpha,\varphi)$, is a $C^*$-subalgebra.

\cite[Theorem~5.7]{Ri2} implies that the action $\alpha$ is proper and saturated with respect to $A_0$, and so by \cite{Ri1} the reduced crossed product $A\times_{\alpha,r} G$ is Morita equivalent to a generalized fixed-point algebra $A^\alpha$. It was shown in \cite[Proposition~3.1]{KQR} that $\fix(A,\alpha,\varphi)$ coincides with the algebra $A^\alpha$ of \cite{Ri1}.

We record a few properties of this averaging process, which can be found in \cite[Section~2]{KQR}:
\begin{itemize}
\item If $a\in A_0$ and $b\in M(A)^\alpha$ then $ab\in A_0$ and $\Phi(ab)=\Phi(a)b$.

\item If $f\in C_c(T)$ and $a\in A_0$ then the function $g\mapsto \varphi(f)\alpha_g(a)$ belongs to $C_c(G,A)$ and
\[
\int_G \varphi(f)\alpha_g(a)\,dg=\varphi(f)\Phi(a).
\]

\item If $(B,\beta)$ is another action, and $\psi:C_0(T)\to M(B)$ and $\sigma:A\to M(B)$ are nondegenerate and $G$-equivariant, 
and if the canonical extension $\bar\sigma:M(A)\to M(B)$ satisfies $\bar\sigma\circ\phi=\psi$,
then $\bar\sigma$ restricts to a nondegenerate homomorphism
\[
\bar\sigma|:\fix(A,\alpha,\varphi)\to M(\fix(B,\beta,\psi)).
\]
Moreover,
\[
\Phi^\beta\circ\sigma=\bar\sigma\circ\Phi^\alpha.
\]

\item For fixed $f,h\in C_c(T)$ the map
\[
a\mapsto \Phi(\varphi(f)a\varphi(h)):A_0\to M(A)
\]
is norm continuous.
\end{itemize}

For the proof of Theorem~\ref{morita} below, we will need to make extensive use of the multiplier bimodules of \cite{EKQR}; the relevant facts are collected in Appendix~\ref{A multiplier}.

\begin{theorem}
\label{morita}
If a locally compact group $G$ acts freely and properly on a topological graph $E$, then $C^*(E)\rtimes_r G$ and $C^*(E/G)$ are strongly Morita equivalent.
\end{theorem}

\begin{proof}
Since $G$ acts freely and properly on $E^0$ and there is a nondegenerate equivariant homomorphism $C_0(E^0)\to M(C^*(E))$, it follows from \cite{Ri2} and Lemma~4.1 in \cite{aHRW} that the corresponding action $\alpha$ of $G$ on $C^*(E)$ is proper and saturated with respect to the $*$-subalgebra $A_0=C_c(E^0) C^*(E)C_c(E^0)$ of $C^*(E)$.
Therefore the reduced crossed product $C^*(E)\rtimes_r G$ is strongly Morita equivalent to the generalized fixed point algebra $C^*(E)^\alpha$.
Thus it suffices to show $C^*(E)^\alpha\cong C^*(E/G)$,
and for this we will construct an injective homomorphism from $C^*(E/G)$ to $M(C^*(E))$ whose image is $C^*(E)^\alpha$.

As before, denote the quotient topological graph by $F=E/G$, so that $F^0=E^0/G$ and $F^1=E^1/G$.
It will also be convenient to denote both quotient maps $E^0\to F^0$ and $E^1\to F^1$ by $q$ (and no confusion will occur since the particular $q$ that is intended will be clear from the context), and to
use the following abbreviations:
\begin{itemize}
\item $A=C_0(E^0)$;

\item $B=C_0(F^0)$;

\item $X=\HH(E)$;

\item $Y=\HH(F)$.
\end{itemize}
Of course, our homomorphism $C^*(F)\to M(C^*(E))$ will be of the form $\tau\times\pi$ for a Cuntz-Pimsner covariant Toeplitz representation $(\tau,\pi)$ of $Y$ in $M(C^*(E))$.

We will construct $(\tau,\pi):Y\to M(C^*(E))$ by, roughly speaking, first mapping the $B$-correspondence $Y$ into the multiplier bimodule $M(X)$, then composing with the canonical Toeplitz representation $(k_X,k_A)$ of the $A$-correspondence $X$ in $C^*(E)$.
Actually, this needs a little tweaking;
our strategy is more accurately indicated by the diagram
\begin{equation}
\label{tau pi}
\xymatrix{
Y \ar@{-->}[r]^-{(\tau,\pi)} \ar[d]_{(\mu,\nu)}
&M(C^*(E))
\\
M_A(X). \ar[ur]_{(\bar{k_X},\bar{k_A})}
}
\end{equation}

Note the appearance in the above diagram of the ``$A$-multiplier'' bimodule $M_A(X)$ (see Definition~\ref{relative multiplier} and Remark~\ref{KM in MK}); this is necessary because the canonical Toeplitz representation $(k_X,k_A)$ can be degenerate --- see Example~\ref{degenerate}.

Our proof is rather long, and to improve readability we break it into the following steps:
\begin{enumerate}
\item construct a correspondence homomorphism $(\mu,\nu)$ as in Diagram~\ref{tau pi};

\item define $(\tau,\pi)$ to make Diagram~\ref{tau pi} commute,
and show that it is a Toeplitz representation;

\item show that the Toeplitz representation $(\tau,\pi)$ is Cuntz-Pimsner covariant;

\item show that the associated homomorphism $\tau\times\pi:C^*(F)\to M(C^*(E))$ is injective;

\item show that the image of $\tau\times\pi$ is $C^*(E)^\alpha$.
\end{enumerate}
We take each step in order:

\begin{step}
Clearly, we should define $(\mu,\nu)$ by composing with the quotient map $q$.
On $B$ this obviously does the right thing:
\[
\nu:B\to M(A), \qquad \nu(f):=f\circ q.
\]

On $Y$, things are a little more delicate:
for $\xi\in C_c(F^1)$ and $f\in C_c(E^0)$ define $\mu(\xi)\cdot f:E^1\to \C$ by
\[
\bigl(\mu(\xi)\cdot f\bigr)(e)=\xi(q(e))f(s(e)).
\]
Then $\mu(\xi)\cdot f\in C_c(E^1)$ by Lemma~\ref{compactly trivial}.

We must show that we can extend $\mu$ to get a correspondence homomorphism $(\mu,\nu):Y\to M_A(X)$,
and we will accomplish this by:
\begin{itemize}
\item 
extending to a correspondence homomorphism into $M(X)$,
and then

\item 
showing that it actually takes values in $M_A(X)$.
\end{itemize}

We first need to know that for $\xi\in C_c(F^1)$ the map
\[
f\mapsto \mu(\xi)\cdot f:C_c(E^0)\to C_c(E^1)
\]
extends to an adjointable map $A\to X$, i.e., an element of $M(X)$.
For technical purposes, we need the following

\textbf{Claim:}
If
$\xi,\eta\in C_c(F^1)$ and $v\in E^0$,
then
\[
\<\xi,\eta\>(q(v))
=\sum_{s(e)=v}\overline{\xi(q(e))}\eta(q(e)).
\]
To see this, just observe that, for every $v'\in F^0$, no matter which element $v\in q\inv(v')$ we choose, the set of values
\[
\bigl\{\overline{\xi(q(e))}\eta(q(e)):s(e)=v\bigr\}
\]
coincides with the set of values
\[
\bigl\{\overline{\xi(e')}\eta(e'):s(e')=v'\bigr\},
\]
because $G$ acts freely and the source map $s:E^1\to E^0$ is $G$-equivariant,
and we have proved the claim.

We use this to show that
the linear map
\[
f\mapsto \mu(\xi)\cdot f:C_c(E^0)\to X
\]
is bounded, and hence
extends uniquely to a bounded linear map
\[
\mu(\xi):A\to X.
\]
For $v\in E^0$ we have
\begin{align*}
\<\mu(\xi)\cdot f,\mu(\xi)\cdot f\>(v)
&=\sum_{s(e)=v}|\mu(\xi)(q(e))|^2|f(v)|^2
\\&=\<\xi,\xi\>(q(e))|f(v)|^2
\\&\le \|\xi\|^2\|f\|^2,
\end{align*}
so
\begin{equation}
\label{bounded}
\|\mu(\xi)\cdot f\|\le \|\xi\|\|f\|,
\end{equation}
verifying the boundedness.

We need to show that the linear map $\mu(\xi):A\to X$ is adjointable.
For $\eta\in C_c(E^1)$ define $\mu(\xi)^*\cdot \eta:E^0\to \C$ by
\[
\bigl(\mu(\xi)^*\cdot \eta\bigr)(v)
=\sum_{s(e)=v}\overline{\xi(q(e))}\eta(e),
\]
where the sum is finite by Corollary~\ref{compactly trivial}.
An argument similar to the proof of \cite[Lemma~1.5]{Ka1} shows that $\mu(\xi)^*\cdot\eta\in C_c(E^0)$.
We have
\begin{align*}
\bigl|\bigl(\mu(\xi)^*\cdot \eta\bigr)(v)\bigr|^2
&\le \sum_{s(e)=v}|\xi(q(e))|^2\sum_{s(e)=v}|\eta(e)|^2
\\&=\<\xi,\xi\>(q(v))\<\eta,\eta\>(v)
\\&\le \|\xi\|^2\|\eta\|^2,
\end{align*}
so
\[
\|\mu(\xi)^*\cdot \eta\|\le \|\xi\|\|\eta\|.
\]
Thus the map
\[
\eta\mapsto \mu(\xi)^*\cdot\eta:C_c(E^1)\to C_c(E^0)
\]
extends uniquely to a bounded linear map
\[
\mu(\xi)^*:X\to A.
\]
For $f\in C_c(E^0)$, $\eta\in C_c(E^1)$, and $v\in E^0$ we have
\begin{align*}
\<\mu(\xi)\cdot f,\eta\>(v)
&=\sum_{s(e)=v}\overline{\xi(q(e))}\overline{f(v)}\eta(e).
\\&=\overline{f(v)}\bigl(\mu(\xi)^*\cdot\eta\bigl)(v),
\end{align*}
so
\[
\<\mu(\xi)\cdot f,\eta\>=\<f,\mu(\xi)^*\cdot \eta\>,
\]
and it follows by continuity that $\mu(\xi)^*:X\to A$ is an adjoint of $\mu(\xi)$.
Therefore $\mu(\xi)\in M(X)$.

We now have a linear map $\mu:C_c(F^1)\to M(X)$, and 
then the estimate~\eqref{bounded} shows $\mu$ is bounded, hence extends uniquely to a bounded linear map $\mu:Y\to M(X)$.
We can now show that the pair $(\mu,\nu)$ is a Toeplitz representation.
First we show that for $\xi,\eta\in Y$ we have
\[
\<\mu(\xi),\mu(\eta)\>=\nu(\<\xi,\eta\>),
\]
and by continuity 
the following computation is sufficient:
for $\xi,\eta\in C_c(F^1)$ and $f\in C_c(E^0)$ and $v\in E^0$ we have
\begin{align*}
\bigl(\<\mu(\xi),\mu(\eta)\>f\bigr)(v)
&=\bigl(\mu(\xi)^*\cdot\mu(\eta)\cdot f\bigr)(v)
\\&=\sum_{s(e)=v}\overline{\xi(q(e))}\bigl(\mu(\eta)\cdot f\bigr)(e)
\\&=\sum_{s(e)=v}\overline{\xi(q(e))}\eta(q(e))f(s(e))
\\&=\<\xi,\eta\>(q(v))f(v)
\\&=\bigl(\nu(\<\xi,\eta\>)f\bigr)(v).
\end{align*}

We next show that for $f\in B$ and $\xi\in Y$ we have
\[
\mu(f\cdot\xi)=\nu(f)\cdot\mu(\xi),
\]
and again by continuity the following computation suffices:
for $f\in C_c(F^0)$, $\xi\in C_c(F^1)$,
$h\in C_c(E^0)$, and $e\in E^1$ we have
\begin{align*}
\bigl(\mu(f\cdot\xi)\cdot h\bigr)(e)
&=(f\cdot\xi)(q(e))h(s(e))
\\&=f(r(q(e)))\xi(q(e))h(s(e))
\\&=f(q(r(e)))\bigl(\mu(\xi)\cdot h\bigr)(e)
\\&=\nu(f)(r(e))\bigl(\mu(\xi)\cdot h\bigr)(e)
\\&=\Bigl(\nu(f)\cdot\bigl(\mu(\xi)\cdot h\bigr)\Bigr)(e)
\\&=\Bigl(\bigl(\nu(f)\cdot\mu(\xi)\bigr)\cdot h\Bigr)(e).
\end{align*}

We thus have a  homomorphism  $(\mu,\nu)$ from  the $B$-correspondence $Y$ to  the $M(A)$-correspondence $M(X)$,
and
it remains to show that for $\xi\in Y$ we actually have $\mu(\xi)\in M_A(X)$, 
i.e., for $f\in A$ we have
\[
f\cdot\mu(\xi)\in X
\]
(because $\mu(\xi)\cdot f\in X$ automatically since $\mu(\xi)$ is a module multiplier).
By continuity it suffices to show that
if $\xi\in C_c(F^1)$
and $f\in C_c(E^0)$
then
\[
f\cdot \mu(\xi)\in C_c(E^1).
\]
For $h\in C_c(E^0)$ and $e\in E^1$ we have
\begin{align*}
\Bigl(\bigl(f\cdot\mu(\xi)\bigr)\cdot h\Bigr)(e)
&=\Bigl(f\cdot\bigl(\mu(\xi)\cdot h\bigr)\Bigr)(e)
\\&=f(r(e))\bigl(\mu(\xi)\cdot h\bigr)(e)
\\&=f(r(e))\xi(q(e))h(s(e)),
\end{align*}
so the multiplier $f\cdot \mu(\xi)$ of $X$ coincides with the element of $C_c(E^1)$ given by
\[
e\mapsto f(r(e))\xi(q(e))
\]
(which has compact support
by Corollary~\ref{compactly trivial}).

\end{step}

\begin{step}
By Corollary~\ref{C=A toeplitz} the Toeplitz representation
\[
(k_X,k_A):X\to C^*(E)
\]
extends to a Toeplitz representation
\[
(\overline{k_X},\overline{k_A}):M_A(X)\to M_A(C^*(E)),
\]
because $k_A$ is nondegenerate.

Now define $(\tau,\pi):Y\to M_A(C^*(E))$ to make diagram~\eqref{tau pi} commute,
i.e.,
\[
\tau:=\bar{k_X}\circ\mu\midtext{and}\pi:=\bar{k_A}\circ\nu.
\]
Then $(\tau,\pi)$ is a Toeplitz representation, being the composition of the correspondence homomorphism $(\mu,\nu)$ and the Toeplitz representation $(\overline{k_X},\overline{k_A})$.
\end{step}

\begin{step}
Cuntz-Pimsner covariance of $(\tau,\pi)$ will follow from an analogous property of the correspondence homomorphism $(\mu,\nu)$.
Our strategy is illustrated by the diagram
\begin{equation}
\label{cuntz pimsner}
\xymatrix{
B_{\rg}
\ar[rr]^-{\nu_{\rg}'} \ar[dd]_{\varphi_{B,\rg}} \ar[dr]^{\pi_{\rg}}
&&M_A(A_{\rg})
\ar[dl]_{\bar{k_{A,\rg}}} \ar[dd]^{\bar{\varphi_{A,\rg}}}
\\
&M_A(C^*(E))
\\
\KK(Y)
\ar[rr]_-{\mu^{(1)}} \ar[ur]^{\tau^{(1)}}
&&M_A(\KK(X))
\ar[ul]_{\bar{k_X^{(1)}}},
}
\end{equation}
which requires explanation.
We define
$B_{\rg}=C_0(F^0_{\rg})$,
and we let $\varphi_{B,\rg}$ denote the restriction to the ideal $B_{\rg}$ of the natural homomorphism $\varphi_B:B\to \LL(Y)$ implementing the left module multiplication,
and similarly for $A_{\rg}$ and $\varphi_{A,\rg}$.
Also similarly for the restrictions $\pi_{\rg}$ of $\pi:B\to M_A(C^*(E))$ and $k_{A,\rg}$ of $k_A:A\to C^*(E)$.
The bars in the right-hand triangle of \eqref{cuntz pimsner} denote the unique $A$-strictly continuous extensions guaranteed by Lemma~\ref{extend C multiplier} and Corollary~\ref{ideal}.
For the bottom arrow, note that the correspondence homomorphism
$\mu:Y\to M_A(X)$ naturally induces a homomorphism
$\mu^{(1)}:\KK(Y)\to \KK(M_A(X))$,
and by Remark~\ref{KM in MK} we have
$\KK(M_A(X))\subset M_A(\KK(X))$.

The part of \eqref{cuntz pimsner} requiring the most explanation is $\nu_{\rg}'$.
We have defined a homomorphism $\nu:B\to M(A)$, and we denote the restriction to $B_{\rg}$ by
\[
\nu_{\rg}:B_{\rg}\to M(A).
\]
For the purposes of the rest of \eqref{cuntz pimsner}, what we need is to have a homomorphism into $M_A(A_{\rg})$.

The problem is that $M_A(A_{\rg})$ doesn't embed naturally in $M(A)$. However, we will show that $\nu_{\rg}$ actually maps into the $C^*$-subalgebra
\[
M_{A,A_{\rg}}(A):=\{f\in M(A):fA\subset A_{\rg}\},
\]
of $M(A)$.
Since every $f\in M_{A,A_{\rg}}(A)$ vanishes on $E^0\minus E^0_{\rg}$,
the restriction map $f\mapsto f|_{E^0_{\rg}}$ gives an embedding of this subalgebra into $M_A(A_{\rg})$,
and we define $\nu_{\rg}'$ so that the diagram
\[
\xymatrix{
B_{\rg} \ar[rr]^-{\nu_{\rg}} \ar[dr]_{\nu_{\rg}'}
&&M_{A,A_{\rg}}(A) \ar@{_(->}[dl]
\\
&M_A(A_{\rg}).
}
\]
commutes.

The desired Cuntz-Pimsner covariance of $(\tau,\pi)$ follows from commutativity of the left triangle of \eqref{cuntz pimsner} (see Remark~\ref{regular}), which we will deduce from that of the outer square and the other three triangles.

\subsection*{The top triangle}
The top triangle of \eqref{cuntz pimsner} coincides with the lower left triangle of the diagram
\begin{equation}\label{nu prime}
\xymatrix@C+30pt@R+30pt{
B_{\rg} \ar[rr]^{\nu_{\rg}} \ar[dr]^{\nu_{\rg}'} \ar[ddr]_{\pi_{\rg}}
&&M_{A,A_{\rg}}(A) \ar@{_(->}[dl] \ar[ddl]^{\bar{k_A}|}
\\
&M_A(A_{\rg}) \ar[d]^(.4){\bar{k_{A,\rg}}}
\\
&M_A(C^*(E)),
}
\end{equation}
in which $\bar{k_A}|$ denotes the restriction of $\bar{k_A}:M(A)\to M_A(C^*(E))$ to the subalgebra $M_{A,A_{\rg}}(A)$.
Once we have verified that $\nu_{\rg}$ maps $B_{\rg}$ into $M_{A,A_{\rg}}(A)$, it will suffice to observe that all the other triangles in \eqref{nu prime} commute:
the top triangle of \eqref{nu prime} commutes by definition of $\nu_{\rg}'$,
the outer triangle commutes by definition of $\pi_{\rg}$, 
and the lower-right triangle commutes by a routine check of the definitions.

We must show that if $f\in C_0(F^0_{\rg})$ and $g\in C_0(E^0)$ then
\[
\nu(f)g=(f\circ q)g\in C_0(E^0_{\rg}).
\]
We have $\nu(f)g\in C_0(E^0)$ since $\nu(f)$ is bounded and $g\in C_0(E^0)$.
We must show that $\nu(f)g=0$ on $E^0\minus E^0_{\rg}$,
equivalently if $\nu(f)g$ is nonzero at an element $v\in E^0$, then we must have $v\in E^0_{\rg}$.
So, we must show that $v$ is in the image of the range map $\tilde r:E^1\to E^0$, and that there is a neighborhood $U$ of $v$ such that $\tilde r\inv(U)$ is compact, i.e., $v\in E^0_{\text{fin}}$.

By Lemma~\ref{husemoller}, we may assume that $E^1=r^*(E^0)$, i.e., that 
the diagram
\[
\xymatrix{
E^0 \ar[d]_q
&E^1 \ar[l]_-{\tilde r} \ar[d]^q
\\
F^0
&F^1 \ar[l]^r
}
\]
coincides with a pull-back
\[
\xymatrix@C+30pt{
E^0 \ar[d]_q
&E^1=r^*(E^0) \ar[l]_-{\pi_2} \ar[d]^{\pi_1}
\\
F^0
&F^1 \ar[l]^r
}
\]

Since $\nu(f)(v)\ne 0$ and $f\in C_0(F^0_{\rg})$, we have $q(v)\in F^0_{\rg}$.
Thus we can choose $e\in F^1$ such that $r(e)=q(v)$.
Then $(e,v)\in E^1=r^*(E^0)$ and $v=\pi_2(e,v)$,
showing that $v$ is in the image of the range map on $E^1$.

To see that $v\in E^0_{\text{fin}}$, start with a compact neighborhood $U$ of $v$ in $E^0$, and put $V=q(U)$, which is a compact neighborhood of $q(v)$ in $F^0$.
Shrink $U$ if necessary so that $W:=r\inv(V)$ is compact in $F^1$.
We will show that $\pi_2\inv(U)$ is compact.
Suppose $\pi_2(e,u)\in U$, i.e., $u\in U$ and $r(e)=q(u)$.
Then $e\in W$.
Thus
\[
(e,u)\in (W\times U)\cap r^*(E^0),
\]
which is compact.
We have verified the inclusion $\nu_{\rg}(B_{\rg})\subset M_{A,A_{\rg}}(A)$.

\subsection*{The right triangle}

Cuntz-Pimsner covariance of the canonical homomorphism $(k_X,k_A)$ is expressed by the commutative diagram
\[
\xymatrix@C+30pt{
&A_{\rg} \ar[dl]_{k_{A,\rg}} \ar[dd]^{\varphi_{A,\rg}}
\\
C^*(E)
\\
&\KK(X) \ar[ul]^{k_X^{(1)}},
}
\]
and it follows from $A$-strict continuity  that the right triangle of \eqref{cuntz pimsner} commutes.

\subsection*{The bottom triangle}

For $\xi,\eta\in Y$ we have
\begin{align*}
\bar{k_X^{(1)}}\bigl(\mu^{(1)}(\theta_{\xi,\eta})\bigr)
&=\bar{k_X^{(1)}}\bigl(\mu(\xi)\mu(\eta)^*\bigr)
\\&=\bar{k_X}(\mu(\xi))\bar{k_X}(\mu(\eta))^*
&&\text{(by Corollary~\ref{C=A toeplitz})}
\\&=\tau(\xi)\tau(\eta)^*
\\&=\tau^{(1)}(\theta_{\xi,\eta}).
\end{align*}

\subsection*{The outer square}

We need to show that if $f\in C_0(F^0_{\rg})$, then
\[
\bar{\varphi_{A,\rg}}\circ \nu_{\rg}'(f)=\mu^{(1)}\circ \varphi_{B,\rg}(f).
\]
As in the commutative diagram~\eqref{nu prime} we can show
\[
\bar{\phi_{A,\rg}}\circ \nu_{\rg}'(f)
=\bar{\phi_A}|\circ \nu_{\rg}(f)
=\bar{\phi_A}(\nu(f)).
\]
Also, $\varphi_{B,\rg}(f)=\varphi_B(f)$.
Thus we must show
\[
\bar{\varphi_A}(\nu(f))=\mu^{(1)}(\varphi_B(f)).
\]

Since $f\circ r\in C_0(F^1)$, we can approximate $f\circ r$ uniformly by a function $h\in C_c(F^1)$, and by \cite[Lemmas~1.15 and 1.16]{Ka1} there are $\xi_i,\eta_i\in C_c(F^1)$ for $1=1\dots n$ such that
\begin{itemize}
\item $\displaystyle h=\sum_1^n\xi_i\bar{\eta_i}$;

\item $\displaystyle\pi_F(h)=\sum_1^n\theta_{\xi_i,\eta_i}$;

\item for each $i$ we have
$\xi_i(e)\eta_i(e')=0$
for all $e\ne e'$ in $F^1$ with $s(e)=s(e')$.
\end{itemize}
For $\zeta\in C_c(E^1)$ and $e\in E^1$ we have
\begin{align*}
\bigl(\mu^{(1)}(\theta_{\xi_i,\eta_i})\zeta\bigr)(e)
&=\bigl(\mu(\xi_i)\mu(\eta_i)^*\zeta\bigr)(e)
\\&=\xi_i(q(e))\bigl(\mu(\eta_i)^*\zeta\bigr)(s(e))
\\&=\xi_i(q(e))\sum_{s(e')=s(e)}\bar{\eta_i(q(e'))}\zeta(e')
\\&=\xi_i(q(e))\sum_{s(q(e'))=s(q(e))}\bar{\eta_i(q(e'))}\zeta(e')
\\&=\xi_i(q(e))\bar{\eta_i(q(e))}\zeta(e).
\end{align*}
Thus
\begin{align*}
\bigl(\mu^{(1)}(\pi_F(h)\zeta\bigr)(e)
&=\sum_1^n\xi_i(q(e))\bar{\eta_i(q(e))}\zeta(e)
\\&=h(q(e))\zeta(e).
\end{align*}
By our choice of $h$ we have
\begin{align*}
\bigl(\nu(f)\cdot \zeta\bigr)(e)
&=\nu(f)(r(e))\zeta(e)
\\&=f\circ q(r(e))\zeta(e)
\\&=f\circ r(q(e))\zeta(e)
\\&\approx h(q(e))\zeta(e).
\end{align*}
Let $K:=\supp\zeta$. The above approximation is uniform over $\zeta$ in 
the subspace
\[
C_K(E^1):=\{\xi\in C_c(E^1):\supp\xi\subset K\},
\]
so 
by the elementary Lemma~\ref{equivalent norm} below
we have
\[
\nu(f)\cdot \zeta\approx \mu^{(1)}(\pi_F(h))\zeta\midtext{in}X.
\]
Now, by definition we have
\[
\nu(f)\cdot \zeta=\bar{\varphi_A}(\nu(f))\zeta,
\]
so we get
\[
\bar{\varphi_A}(\nu(f))\zeta\approx \mu^{(1)}(\pi_F(h))\zeta.
\]
Since $f\circ r\approx h$, we have $\varphi_B(f)\approx \pi_F(h)$, so
\begin{align*}
\mu^{(1)}(\varphi_B(f))\zeta
&\approx \mu^{(1)}(\pi_F(h))\zeta
\\&\approx \bar{\varphi_A}(\nu(f))\zeta.
\end{align*}
Thus $\|\mu^{(1)}(\varphi_B(f))\zeta-\bar{\varphi_A}(\nu(f))\zeta\|$ is arbitrarily small, so 
we must have
\[
\mu^{(1)}(\varphi_B(f))\zeta=\bar{\varphi_A}(\nu(f))\zeta.
\]
Since 
$\zeta$ was an arbitrary element of $C_c(E^1)$, which is dense in $X$, we have shown that the outer square commutes.

In the above argument we used the following lemma:

\begin{lem}\label{equivalent norm}
For any compact subset $K\subset E^1$,
on the subspace $C_K(E^1)$ of $X$ the Hilbert-module norm from $X$ is equivalent to the uniform norm.
\end{lem}

\begin{proof}
The uniform norm on $C_K(E^1)$ is of course smaller than the Hilbert-module norm from $X$.
On the other hand, by Lemma~\ref{compactly trivial} (ii) we can let $d$ be an upper bound for the cardinalities of the intersections $K\cap s\inv(v)$ for $v\in E^0$. Then for any $v\in E^0$ we have
\begin{align*}
\sum_{s(e)=v}|\xi(e)|^2
\le d\|\xi\|_u^2,
\end{align*}
where $\|\cdot\|_u$ denotes the uniform norm.
\end{proof}

It now follows that the left triangle of \eqref{cuntz pimsner} commutes, and hence that the Toeplitz representation $(\tau,\pi):Y\to M(C^*(E))$ is Cuntz-Pimsner covariant.

\end{step}

\begin{step}
To see that $\tau\times\pi:C^*(F)\to M(C^*(E))$ is injective,  we apply the Gauge-Invariant Uniqueness Theorem \cite[Theorem~4.5]{Ka1}. For this we need to show that:
\begin{enumerate}
\item $\pi:B\to M(C^*(E))$ is injective, and

\item for all $z\in\T$ there is an automorphism $\sigma_z$ of $(\tau\times\pi)(C^*(F))$ such that
$\sigma_z(\tau(\xi))=z\tau(\xi)$ for $\xi\in C_c(F^1)$ and
$\sigma_z(\pi(f))=\pi(f)$ for $f\in C_c(F^0)$.
\end{enumerate}

For (1), just note that both $\nu:B\to M(A)$ and $\bar{k_A}:M(A)\to M(C^*(E))$ are injective (the latter because $k_A:A\to C^*(E)$ is), and hence so is $\bar{k_A}\circ\nu$.

For (2), we extend the gauge automorphism $\gamma_z$ of $C^*(E)$ to an automorphism $\bar{\gamma_z}$ of the multiplier algebra $M(C^*(E))$, and note that by strict and $A$-strict continuity we have
\begin{itemize}
\item $\bar{\gamma_z}\circ \bar{k_X}=z\bar{k_X}$ on $M_A(X)$, and

\item $\bar{\gamma_z}\circ \bar{k_A}=\bar{k_A}$ on $M(A)$.
\end{itemize}
\end{step}

\begin{step}
Finally, we need to show that the image $(\tau\times\pi)(C^*(F))$ coincides with the generalized fixed-point algebra of $C^*(E)$ under the action $\alpha$ of $G$.
We have an action $\beta$ of $G$ on $C_0(E^0)$ corresponding to the free and proper action on $E^0$, and a $G$-equivariant nondegenerate homomorphism $k_A:A=C_0(E^0)\to C^*(E)\subset M(C^*(E))$, so we can form the generalized fixed-point algebra
\[
C^*(E)^\alpha=
\fix(C^*(E),\alpha,k_A),
\]
and we will show that it coincides with $(\tau\times\pi)(C^*(F))$.
We begin with the inclusion
\[
\fix(C^*(E),\alpha,k_A)\subset(\tau\times\pi)(C^*(F)),
\]
which will occupy us for some time.

For $n\in\N=\{0,1,2,\dots\}$ let $E^n$ denote the space of paths in $E$ of length $n$ (with the relative product topology). Then by \cite{Ka1}, $E^n$ is naturally a topological graph with vertex space $E^0$, and the associated $A$-correspondence $\HH(E^n)$ is isomorphic to the $n$-fold balanced tensor product
\[
X\otimes_A\cdots\otimes_A X,
\]
which we denote simply by $X^n$.
We let $k_X^n:X^n\to C^*(E)$ be the representation corresponding  to the $n$-fold tensor power of $k_X$.
Then
\[
C^*(E)=\clspn\{k_X^n(C_c(E^n))k_X^m(C_c(E^m))^*:n,m\in\N\}.
\]
Our hypotheses on $E$, namely that $G$ acts freely and properly, carry over to $E^n$.
We use the same notational conventions as for the quotient graph $F=E/G$,
and we write $\tau^n=\bar{k_X^n}\circ\mu$.

Now let $f,h\in C_c(E^0)$ and $a\in C^*(E)$. We must show that
\[
\Phi\bigl(k_A(f)ak_A(h)\bigr)\in (\tau\times\pi)(C^*(F)).
\]
We may assume that
\[
a=k_X^n(\xi)k_Y^m(\eta)^*,
\]
where $\xi\in C_c(E^n)$ and $\eta\in C_c(E^m)$.
By the elementary Lemma~\ref{factor} below, 
we can find 
finitely many functions
$\xi_i\in C_c(F^n),
\eta_j\in C_c(F^m),
f_i,g_j\in C_c(E^0)$ such that
\[
f\cdot \xi\approx \sum_i \mu(\xi_i)\cdot f_i
\midtext{and}
\eta\approx \sum_j g_j\cdot \mu(\eta_j).
\]
Then for each $i,j$, again by Lemma~\ref{factor} there exist finitely many
$f_{ijl}\in C_c(E^0),\eta_{ijl}\in C_c(F^m)$ such that
\[
\mu(\eta_j)\cdot \bar{f_i}\approx \sum_l f_{ijl}\cdot \mu(\eta_{ijl}).
\]
We have
\begin{align*}
k_A(f)a
&=k_A(f)k_X^n(\xi)k_X^m(\eta)^*
\\&=k_X^n(f\cdot \xi)k_X^m(\eta)^*
\\&\approx \sum_{ij} k_X^n\bigl(\mu(\xi_i)\cdot f_i\bigr)k_X^m\bigl(g_j\cdot \mu(\eta_j)\bigr)^*
\\&=\sum_{ij}\bar{k_X^n}(\mu(\xi_i))k_A(f_i)\bar{k_X^m}(\mu(\eta_j))^*k_A(\bar{g_j})
\\&=\sum_{ij}\bar{k_X^n}(\mu(\xi_i))\bigl(\bar{k_X^m}(\mu(\eta_j))k_A(\bar{f_i})\bigr)^*k_A(\bar{g_j})
\\&=\sum_{ij}\bar{k_X^n}(\mu(\xi_i))\Bigl(k_X^m\bigl(\mu(\eta_j)\cdot \bar{f_i}\bigr)\Bigr)^*k_A(\bar{g_j})
\\&\approx \sum_{ijl}\bar{k_X^n}(\mu(\xi_i))k_X^m\bigl(f_{ijl}\cdot \mu(\eta_{ijl})\bigr)^*k_A(\bar{g_j})
\\&=\sum_{ijl}\bar{k_X^n}(\mu(\xi_i))\tau^m(\eta_{ijl})^*k_A(\bar{f_{ijl}g_j}).
\end{align*}
Now, we can choose $h'\in C_c(E^0)$ such that
\[
h'f=f\midtext{and}h'\cdot \mu(\xi_i)=\mu(\xi_i)\midtext{for all}i.
\]
Then we have
\begin{align*}
\Phi\bigl(k_A(f)ak_A(h)\bigr)
&=\Phi\bigl(k_A(h')k_A(f)ak_A(h)\bigr)
\\&\approx\sum_{ijl}\Phi\bigl(k_A(h')\bar{k_X^n}(\mu(\xi_i))\tau^m(\eta_{ijl})^*k_A(\bar{f_{ijl}g_j})\bigr)
\\&=\sum_{ijl}\Phi\bigl(\bar{k_X^n}(\mu(\xi_i))\tau^m(\eta_{ijl})^*k_A(\bar{f_{ijl}g_j})\bigr)
\\&=\sum_{ijl}\tau^n(\xi_i)\tau^m(\eta_{ijl})^*\Phi\bigl(k_A(\bar{f_{ijl}g_j})\bigr).
\end{align*}
Since $\tau^n(\xi_i)\tau^m(\eta_{ijl})^*\in (\tau\times\pi)(C^*(F))$ and
\begin{align*}
\Phi\circ k_A(C_c(E^0))
&=\bar{k_A}\circ\Phi^\beta(C_c(E^0))
\\&\subset \bar{k_A}\circ\nu(C_c(F^0))
\\&=\pi(C_c(F^0))
\\&\subset (\tau\times\pi)(C^*(F)),
\end{align*}
it follows that $\Phi(k_A(f)ak_A(h))\in (\tau\times\pi)(C^*(F))$.

The above argument used the following elementary lemma:

\begin{lem}\label{factor}
In the norm of $X$,
\[
C_c(E^1)\subset\clspn\bigl\{\mu(C_c(F^1))\cdot C_c(E^0)\bigr\}
\cap\clspn\bigl\{C_c(E^0)\cdot \mu(C_c(F^1))\bigr\}.
\]
\end{lem}

\begin{proof}
By Theorem~\ref{structure} we may assume that $E^1=s^*(E^0)$.
Let $\theta:s^*(E^0)\to r^*(E^0)$ be the isomorphism of Theorem~\ref{structure}.
Note that
\[
\mu(\eta)\cdot f=(\eta\otimes f)|_{s^*(E^0)}
\midtext{and}
f\cdot \mu(\eta)=(\eta\otimes f)\circ\theta,
\]
where $(\eta\otimes f)(e,v)=\eta(e)f(v)$ for $(e,v)\in F^1\times E^0$.

Let $\xi\in C_c(E^1)$.
Then there is $\xi'\in C_c(F^1\times E^0)$ such that
\[
\xi=\xi'|_{s^*(E^0)},
\]
because $s^*(E^0)$ is a closed subset of $F^1\times E^0$.
Then there are finitely many $\eta_i\in C_c(F^1),f_i\in C_c(E^0)$ such that
\[
\xi'\approx\sum_i(\eta_i\otimes f_i).
\]
Thus
\[
\xi\approx\sum_i(\eta_i\otimes f_i)|_{s^*(E^0)}=\sum_i\mu(\eta_i)\cdot f_i.
\]
This shows the  inclusion in the first set.
For the other one,
since $\xi\circ\theta\inv\in C_c(r^*(E^0))$,
there is $\zeta\in C_c(F^1\times E^0)$ such that
\[
\xi\circ\theta\inv=\zeta|_{r^*(E^0)},
\]
and then as above we can approximate
\[
\zeta\approx\sum_i(\eta_i\otimes f_i),
\]
so that
\[
\xi\circ\theta\inv\approx\sum_i(\eta_i\otimes f_i)|_{r^*(E^0)},
\]
and hence
\[
\xi\approx\sum_i(\eta_i\otimes f_i)\circ\theta=\sum_if_i\cdot\mu(\eta_i).
\qedhere
\]
\end{proof}

We have shown that
$\fix(C^*(E),\alpha,k_A)\subset(\tau\times\pi)(C^*(F))$,
and we turn to the opposite inclusion.
We need the following easy lemma.

\begin{lemma}
\label{lift}
For all $f\in C_c(F^0)$ there exists $h\in C_c(E^0)$ such that
$\nu(f)=\Phi^\beta(h)$.
\end{lemma}

\begin{proof}
Put $M=\supp f$, a compact subset of $F^0$.
Since the quotient map $q:E^0\to F^0$ is a continuous open surjection, it is a standard fact that there exists a compact set $K\subset E^0$ such that $q(K)=M$. Choose $h_1\in C_c(E^0)$ such that $h_1(v)\ne 0$ for all $v\in K$. Then $\Phi^\beta(h_1)\ne 0$ on $q\inv(M)$, and there is a unique $h_2\in C_c(F^0)$ such that $\Phi^\beta(h_1)=\nu(h_2)$.
Then $h_2\ne 0$ on $M$, so there exists $h_3\in C_c(F^0)$ such that
\[
h_3=\frac{f}{h_2}\midtext{on}M.
\]
Then $f=h_2h_3$, so
\begin{align*}
\nu(f)
&=\nu(h_2)\nu(h_3)
\\&=\Phi^\beta(h_1)\nu(h_3)
\\&=\Phi^\beta\bigl(h_1\nu(h_3)\bigr),
\end{align*}
so we can take $h=h_1\nu(h_3)$.
\end{proof}

To finish, it suffices to show that $\tau(\xi)$ and $\pi(f)$ are in $\fix(C^*(E),\alpha,k_A)$ for all $\xi\in C_c(F^1)$ and $f\in C_c(F^0)$.
For $\xi$, choose $h\in C_c(F^0)$ such that $\xi=\xi\cdot h$.
By Lemma~\ref{lift} we can choose $h_1\in C_c(E^0)$ such that $\nu(h)=\Phi^\beta(h_1)$.
We have
\begin{align*}
\tau(\xi)
&=\tau(\xi\cdot h)
\\&=\tau(\xi)\pi(h)
\\&=\tau(\xi)\bar{k_A}(\nu(h))
\\&=\tau(\xi)\bar{k_A}(\Phi^\beta(h_1))
\\&=\tau(\xi)\Phi(k_A(h_1))
\\&=\Phi\bigl(\tau(\xi)k_A(h_1)\bigr)
\\&=\Phi\bigl(\bar{k_X}(\mu(\xi))k_A(h_1)\bigr)
\\&=\Phi\bigl(k_X(\mu(\xi)\cdot h_1)\bigr).
\end{align*}
Now, by Corollary~\ref{compactly trivial} we have $\mu(\xi)\cdot h_1\in C_c(E^1)$, so we can choose $h_2\in C_c(E^0)$ such that
\[
\mu(\xi)\cdot h_1=h_2\cdot \mu(\xi)\cdot h_1,
\]
and then we can factor $h_2=h_3h_4$ with $h_3,h_4\in C_c(E^0)$. Then
\begin{align*}
\tau(\xi)
&=\Phi\Bigl(k_X\bigl(h_3h_4\cdot \mu(\xi)\cdot h_1\bigr)\Bigr)
\\&=\Phi\Bigl(k_A(h_3)k_X\bigl(h_4\cdot \mu(\xi)\bigr)k_A(h_1)\Bigr).
\end{align*}
Since $h_4\cdot \mu(\xi)\in C_c(E^1)$ (again by Corollary~\ref{compactly trivial}), we have
\[
k_X\bigl(h_4\cdot \mu(\xi)\bigr)\in C^*(E),
\]
and therefore $\tau(\xi)\in \fix(C^*(E),\alpha,k_A)$.

Our strategy for $f\in C_c(F^0)$ is similar, but things are somewhat easier this time: factor $f=fh_1$ with $h_1\in C_c(F^0)$, then choose $h_2\in C_c(E^0)$ such that $\nu(h_1)=\Phi^\beta(h_2)$, then factor $h_2=h_3h_4h_5$ with $h_i\in C_c(E^0)$. We have
\begin{align*}
\pi(f)
&=\pi(fh_1)
\\&=\pi(f)\pi(h_1)
\\&=\pi(f)\bar{k_A}(\nu(h_1))
\\&=\pi(f)\bar{k_A}(\Phi^\beta(h_2))
\\&=\pi(f)\Phi(k_A(h_3h_4h_5))
\\&=\Phi\bigl(\pi(f)k_A(h_3h_4h_5)\bigr)
\\&=\Phi\bigl(\bar{k_A}(\nu(f))k_A(h_3h_4h_5)\bigr)
\\&=\Phi\bigl(k_A(h_3)k_A(fh_3)k_A(h_5)\bigr),
\end{align*}
which is in $\fix(C^*(E),\alpha,k_A)$ because $k_A(fh_3)\in C^*(E)$.
\qedhere
\end{step}
\end{proof}

Let $E$ be a topological graph and let $c:E^1\to G$ be a continuous map. Recall that $G$ acts freely and properly on $G\times_cE$.  Since $(G\times_cE)/G \cong E$,
we have by Theorem \ref{morita}:

\begin{corollary}
The $C^*$-algebras $C^*(E\times_cG)\rtimes_r G$ and  $C^*(E)$ are strongly Morita equivalent. In particular, for a finitely generated locally compact group $G$ with generators $S=\{h_1,h_2,...,h_n\}$ and Cayley graph $E(G,S)$, we get that $C^*(E(G,S))\rtimes_r G$ is strongly Morita equivalent to the Cuntz algebra ${\mathcal O}_n$.
\end{corollary}

\begin{remark}
 If $G$ is abelian, there is an action $\alpha^c$ of $\hat{G}$ on $C^*(E)$ such that
\[(\alpha^c_\chi\xi)(e)=\langle \chi, c(e)\rangle \xi(e)\]
for $\xi\in C_c(E^1)$ and $\chi\in\hat{G}$.
Then \[C^*(E)\rtimes_{\alpha^c} \hat{G}\cong C^*(E\times_cG).\]
For $G$ non-abelian, we need to use the notion of coaction.
This will be investigated in \cite{KQR2}.
\end{remark}

\begin{appendix}

\section{multipliers of correspondences}
\label{A multiplier}

In Section~\ref{actions}
we make extensive use of the \emph{multiplier bimodules} of
\cite[Chapter~1]{EKQR}. The bimodules in \cite{EKQR} were called
\emph{right-Hilbert $A-B$ bimodules}, which are also commonly referred
to as \emph{$A-B$ correspondences}, where $A$ and $B$ are
$C^*$-algebras. In most of this paper we are interested in
correspondences over a single $C^*$-algebra, but occasionally we need
the added generality of allowing $A$ and $B$ to be possibly different.
Warning: there is a nondegeneracy issue; in the literature, $A-B$
correspondences are often (usually?) allowed to be degenerate, in the
sense that the closed span of $A\cdot X$ might be smaller than $X$,
but in \cite{EKQR} the right-Hilbert bimodules are required to be
nondegenerate. This will cause no problem here, though, because the
correspondences associated to topological graphs are all
nondegenerate, so that the results of \cite{EKQR} may be applied
freely.

To emphasize, throughout this appendix we will assume that all our correspondences are nondegenerate.

$A$ and $B$ are the \emph{coefficient algebras} of the correspondence
$X$, and we will sometimes denote the correspondence by ${}_AX_B$ to
indicate what the coefficient algebras are. Note that a Hilbert
$B$-module is also a $\C-B$ correspondence.

We recall briefly the theory of multiplier bimodules; see
\cite[Chapter~1]{EKQR} for more details. A \emph{multiplier} of 
${}_AX_B$ is by definition an adjointable map from $B$ to $X$, and
$M(X)$ denotes the set of all multipliers of $X$. Thus by definition
we have $M(X)=\LL(B,X)$. Note that the definition of $M(X)$ does not
involve the left $A$-module structure at all; it is, in
essence, a Hilbert $B$-module concept. However, $M(X)$ is an
$\LL(X)-M(B)$ correspondence, hence an $M(A)-M(B)$ correspondence
because of our assumption that $X$ is nondegenerate
as a left $A$-module. The restriction of the right $M(B)$-module
action to $B$ is just the evaluation of adjointable maps; i.e., if
$m\in M(X)$ and $b\in B$ then we write $m\cdot b$ to mean the value of
the adjointable map $m:B\to X$ at the element $b$ of its domain. There
is an embedding of $X$ in $M(X)$ such that the adjointable map
associated to $\xi\in X$ is
\[
b\mapsto \xi\cdot b\midtext{for}b\in B.
\]
$M(X)$ has both an $\LL(X)$-valued inner product
\[
{}_{\LL(X)}\<m,n\>=mn^*\midtext{for}m,n\in M(X)
\]
and an $M(B)$-valued inner product
\[
\<m,n\>_{M(B)}=\<m,n\>=m^*n.
\]
If either $m$ or $n$ lies in $X$, then
\[
{}_{\LL(X)}\<m,n\>\in \KK(X)\midtext{and}\<m,n\>\in B,
\]
because, more generally, composing an adjointable operator on either
side with a compact operator gives a compact operator. The restriction
of the left inner product ${}_{\LL(X)}\<\cdot,\cdot\>$ to $X$ gives
the rank-one operators:
\[
{}_{\LL(X)}\<\xi,\eta\>={}_{\KK(X)}\<\xi,\eta\>=\theta_{\xi,\eta},
\]
and the restriction of the right inner product to $X$ is the given
one. The left and right module actions of multipliers of the
coefficient algebras $A$ and $B$, respectively, on an element $m\in
M(X)$ are as follows: if $n\in M(A)$ then $n\cdot m$ is the element of
$M(X)$ defined by
\[
(n\cdot m)\cdot b=n\cdot (m\cdot b)\midtext{for}b\in B,
\]
and if $n\in M(B)$ then $m\cdot n$ is defined by
\[
(m\cdot n)\cdot b=m\cdot (nb)\midtext{for}b\in B.
\]
Of course, if $n\in B$ then $m\cdot n\in X$.

The \emph{strict topology} on $M(X)$ is generated by the seminorms
\[
m\mapsto \|T\cdot m\|\midtext{and}m\mapsto \|m\cdot b\|\midtext{for}T\in\KK(X),b\in B.
\]
The operations on the $M(A)-M(B)$ correspondence $M(X)$ are separately strictly continuous, and
$M(X)$ is the strict completion of $X$.

Given another correspondence ${}_CY_D$
and homomorphisms $\pi:A\to M(C)$ and $\theta:B\to M(D)$,
a linear map $\psi:X\to M(Y)$ is a \emph{$\pi-\theta$ compatible correspondence homomorphism} if the two conditions
\begin{align*}
\<\psi(\xi),\psi(\eta)\>_{M(D)}&=\theta(\<\xi,\eta\>_B);
\\
\psi(a\cdot \xi)&=\pi(a)\cdot \psi(\xi)
\end{align*}
hold.
Note that $\psi$ will automatically satisfy
\[
\psi(\xi\cdot b)=\psi(\xi)\theta(b).
\]
$\pi$ and $\theta$ are the \emph{coefficient} homomorphisms of $\psi$,
and sometimes we write ${}_\pi\psi_\theta$ for the correspondence
homomorphism to indicate what the coefficient homomorphisms are.

A correspondence homomorphism ${}_\pi\psi_\theta$ is \emph{nondegenerate} if $\pi$ and $\theta$ are nondegenerate and
\[
\clspn\{\psi(X)\cdot D\}=Y,
\]
in which case ${}_\pi\psi_\theta$ extends uniquely to a strictly
continuous homomorphism
\[
{}_{\bar\pi}\bar\psi_{\bar\theta}:
{}_{M(A)}M(X)_{M(B)}\to {}_{M(C)}M(Y)_{M(D)}.
\]
(In fact, as for $C^*$-algebras, the extension is unique as just a correspondence homomorphism, and is automatically strictly continuous, although we will not need this fact.)
If $A=B$, $C=D$, and $\pi=\theta$ we will write
\[
(\psi,\pi)={}_\pi\psi_\theta,
\]
and if
in addition $Y=C$ is a correspondence over itself in the canonical
way, and if both $\psi$ and $\pi$ map into $C$ (and not just $M(C)$),
then $(\psi,\pi)$ reduces to the familiar concept of a \emph{Toeplitz
representation} of the $A$-correspondence $X$ into $C$.

If
\[
{}_\pi\psi_\theta:{}_AX_B\to M({}_CY_D)
\]
is a nondegenerate homomorphism, then there is a unique nondegenerate homomorphism
\[
\psi^{(1)}:\KK(X)\to M(\KK(Y))=\LL(Y)
\]
such that
\[
\psi^{(1)}(\theta_{\xi,\eta})=\psi(\xi)\psi(\eta)^*\midtext{for}\xi,\eta\in X,
\]
and then, as usual, by nondegeneracy $\psi^{(1)}$ extends uniquely to a homomorphism
\[
\bar{\psi^{(1)}}:\LL(X)\to \LL(Y).
\]
Note that even if $(\psi,\pi)$ is degenerate, there is still a homomorphism $\psi^{(1)}:\KK(X)\to \KK(M(Y))\subset M(\KK(Y))$, but then there may not be an extension $\bar{\psi^{(1)}}$ to $\LL(X)$ (although we will define a partial
extension, under additional hypotheses, in Corollary~\ref{C=A toeplitz}).

\subsection*{$C$-multipliers}

Unfortunately, we must deal with possibly degenerate homomorphisms of correspondences, and for this reason we have been lead to develop a generalization of the
``$C$-multipliers'' of \cite[Section~1.4]{EKQR}. 

\begin{example}
\label{degenerate}
We can see why degeneracy is an issue for us already in the case of (discrete) directed graphs: let $E$ be the directed graph with a single non-loop edge, i.e.,
\[
\xymatrix{
u\bullet
&\bullet v
\ar[l]_e
}
\]
We have:
\begin{itemize}
\item $C^*(E)=M_2$;

\item $X=\HH(E)=\C\cdot \Chi_{\{e\}}$;

\item $k_X(\Chi_{\{e\}})=\smtx{0&1\\0&0}$;

\item $k_X(X)C^*(E)=\smtx{*&*\\0&0}$,
\end{itemize}
so the canonical Toeplitz representation $(k_X,k_A)$ is degenerate (where $A=\C^2$).
\end{example}

\begin{remark}
Of course, there are situations in which degeneracy is not an issue; for example,
if the
correspondence $X$ is actually an imprimitivity bimodule (so that
$\varphi_A:A\to \KK(X)$ is an isomorphism), then every Cuntz-Pimsner covariant Toeplitz
representation $(\psi,\pi)$ of $X$ to a $C^*$-algebra $B$ for which $\pi$ is
nondegenerate is automatically nondegenerate as a correspondence
homomorphism, i.e., $\psi(X)B=B$, by \cite[Lemma~5.1]{KQR3}.
Note also that whenever $E$ is a topological graph with $s$ and $r$ both surjective and
$r$ proper, then $k_X:X\to C^*(E)$ will be nondegenerate, hence will extend uniquely to the multiplier bimodule $M(X)$.
\end{remark}

The following
is similar to that in \cite{EKQR}.
To prepare for the general framework introduced below, we begin with some prefatory comments.
First, to establish context, we recall that the ``$C$-multipliers'' in \cite{EKQR} involved tensor products: if $A$ and $C$ are $C^*$-algebras, then \cite{EKQR} defined $M_C(A\otimes C)$ to be all multipliers of $A\otimes C$ that multiply $1\otimes C$ into $A\otimes C$. Here we need a generalization to situations where there are no tensor products.


\begin{defn}
Let $\kappa:C\to M(A)$ be a nondegenerate homomorphism.
\begin{enumerate}
\item A \emph{$C$-multiplier} of $A$ is a multiplier $m\in M(A)$ such that
\[
\kappa(C) m\cup m\kappa(C)\subset A,
\]
and $M_C(A)$ denotes the set of all $C$-multipliers of $A$.

\item The \emph{$C$-strict} topology on $M_C(A)$ is generated by the seminorms
\[
m\mapsto \|\kappa(c) m\|\midtext{and}\mapsto \|m\kappa(c)\|\midtext{for}c\in C.
\]
\end{enumerate}
\end{defn}

Often $C$ will be a nondegenerate $C^*$-subalgebra of $M(A)$ and $\kappa$ will be the inclusion map.

The following generalizes \cite[Proposition~A.5]{EKQR}:

\begin{lemma}
\label{C multiplier}
\

\begin{enumerate}
\item The $C$-strict topology on $M_C(A)$ is stronger than the relative strict topology from $M(A)$.

\item $M_C(A)$ is a $C^*$-subalgebra of $M(A)$, and multiplication and involution are separately $C$-strictly continuous on $M_C(A)$.

\item $M_C(A)$ is the $C$-strict completion of $A$;

\item $M_C(A)$ is an $M(C)$-subbimodule of $M(A)$.
\end{enumerate}
\end{lemma}

\begin{proof}
(1)
Let $m_i\to 0$ $C$-strictly in $M_C(A)$.
We must show that if $a\in A$ then $\|am_i\|$ and $\|m_ia\|$ both tend to $0$.
By the Hewitt-Cohen factorization theorem (see, for example, \cite[Proposition~2.33]{RW}) we can factor $a=a'\kappa(c)$
for some $a'\in A$ and $c\in C$, and then
\[
\|am_i\|=\|a'\kappa(c)m_i\|
\le \|a'\|\|\kappa(c)m_i\|
\to 0,
\]
and similarly for $m_i a$.

(2)
It follows straight from the definitions that $M_C(A)$ is a $C^*$-subalgebra.
Let $m_i\to 0$ $C$-strictly in $M_C(A)$, and let $n\in M(A)$. We must show that $m_in\to 0$ $C$-strictly in $M_C(A)$, and it will follow similarly that $nm_i\to 0$ $C$-strictly.
For $c\in C$ we have
\[
\|\kappa(c)m_in\|\le \|\kappa(c)m_i\|\|n\|\to 0
\]
because $m_i\to 0$ $C$-strictly, and
\[
\|m_in\kappa(c)\|\to 0
\]
because $n\kappa(c)\in A$ and $m_i\to 0$ strictly by (1).

It is even easier to verify that $m_i^*\to 0$ $C$-strictly.

(3)
We first show that $A$ is $C$-strictly dense in $M_C(A)$.
Let $m\in M_C(A)$, and let $\{e_i\}$ be an approximate identity for $C$.
We will show that $\kappa(e_i)m\to m$ $C$-strictly.
For $c\in C$ we have
\[
\|\kappa(c) (\kappa(e_i) m-m)\|
=\|\kappa(ce_i-c)\cdot m\|
\le \|ce_i-c\|\|m\|
\to 0,
\]
while
\[
\|(\kappa(e_i)m-m)\kappa(c)\|
=\|\kappa(e_i)m\kappa(c)-m\kappa(c)\|
\to 0
\]
because $\kappa(e_i)\to 0$ strictly in $M(A)$.

Now we verify the $C$-strict completeness.
Let $\{m_i\}$ be a $C$-strictly Cauchy net in $M_C(A)$.
Then it is also strictly Cauchy,
so converges strictly to some $m\in M(A)$ by strict completeness.
We will show that $m\in M_C(A)$ and $m_i\to m$ $C$-strictly.
Let $\epsilon>0$ and $c\in C$.
Choose $i_0$ such that
\[
\|\kappa(c) (m_i-m_j)\|\le \epsilon
\midtext{for all}i,j\ge i_0.
\]
Then for all $a\in A$ with $\|a\|\le 1$ we have $\|(m_j-m)a\|\to 0$,
so for all $i\ge i_0$ we have
\[
\|\kappa(c) (m_i-m)a\|
=\lim_j\|\kappa(c) (m_i-m_j)a\|
\le \limsup_j\|\kappa(c) (m_i-m_j)\|
\le \epsilon,
\]
and hence $\|\kappa(c) (m_i-m)\|\le \epsilon$.
It follows that $\|\kappa(c) (m_i-m)\|\to 0$.
Since $\kappa(c) m_i\in A$ for all $i$,
we have $\kappa(c) m\in A$.
Similarly, $\|(m_i-m)\kappa(c)\|\to 0$
and $m\kappa(c)\in A$.
Thus $m\in M_C(A)$,
and the above arguments then show that $m_i\to m$ $C$-strictly.

(4)
We have
\[
\bar\kappa(M(C))M_C(A)\kappa(C)\subset \bar\kappa(M(C))A\subset A
\]
and
\[
\kappa(C)\bar\kappa(M(C))M_C(A)\subset \kappa(C)M_C(A)\subset A,
\]
so $\bar\kappa(M(C))M_C(A)\subset M_C(A)$, and similarly for $M_C(A)\bar\kappa(M(C))$.
\end{proof}

The whole point of $C$-multipliers is to extend degenerate homomorphisms.
The following generalizes \cite[Proposition~A.6]{EKQR}:

\begin{lemma}
\label{extend C multiplier}
Let 
$\kappa:C\to M(A)$,
$\sigma:D\to M(B)$,
$\pi:A\to M_D(B)$,
and $\lambda:C\to M(\sigma(D))$
be homomorphisms,
with $\kappa$, $\sigma$, and $\lambda$ nondegenerate,
such that
\[
\pi(\kappa(c)a)=\lambda(c)\pi(a)\midtext{for}c\in C,a\in A.
\]
Then $\pi$ extends uniquely to a $C$-strict to $D$-strictly continuous homomorphism $\bar\pi:M_C(A)\to M_D(B)$.

Moreover, for $n\in M(C)$ and $m\in M_C(A)$ we have
\[
\bar\pi(\bar\kappa(n)m)=\bar\lambda(n)\bar\pi(m)
\midtext{and}
\bar\pi(m\bar\kappa(n))=\bar\pi(m)\bar\lambda(n).
\]
\end{lemma}

Note that, in the above lemma, $M(\sigma(D))$ is identified with the idealizer of $\sigma(D)$ in $M(B)$ (and this is valid since $\sigma$ is nondegenerate), and
$\bar\lambda$ and $\bar\kappa$ denote the canonical extensions to the multiplier algebras (which exist by nondegeneracy).

\begin{proof}
We first show that $\pi$ is $C$-strict to $D$-strictly continuous.
Let $m_i\to 0$ $C$-strictly in $A$,
and let $d\in D$.
By the Hewitt-Cohen factorization theorem we can factor
$\sigma(d)=\sigma(d')\lambda(c)$
for some $d'\in D$ and $c\in C$, and then
\begin{align*}
\|\sigma(d)\pi(m_i)\|
&=\|\sigma(d')\lambda(c)\pi(m_i)\|
\\&=\|\sigma(d')\pi(\kappa(c)m_i)\|
\\&\le \|d'\|\|\kappa(c)m_i\|
\\&\to 0,
\end{align*}
and similarly for $\pi(m_i)\sigma(d)$.

Thus, by completeness there is a unique $C$-strict to $D$-strictly continuous linear extension $\bar\pi:M_C(A)\to M_D(B)$.
Since the algebraic operations on $M_C(A)$ and $M_D(B)$ are separately continuous for the $C$-strict and $D$-strict topologies, respectively,
$\bar\pi$ is a $*$-homomorphism.

The other part follows from continuity and density.
\end{proof}

\begin{remark}
(1)
In fact, the above extension $\bar\pi$ is unique as a homomorphism, but we will not need this.

(2)
The above lemma could be the starting point for an investigation of what one might call ``decorated $C^*$-algebras'' (where $A$ is ``decorated'' by $C$ and $B$ is decorated by $D$). However, we will not pursue this any further here.
\end{remark}

One of our applications of the above result is described by the following special case, in which we will actually need $\pi$ to take values in $B$ itself:

\begin{corollary}
\label{ideal}
Let $I$ be an ideal of $A$, 
and let $\pi:A\to B$ be a nondegenerate homomorphism.
Then the restriction $\pi_I$ of $\pi$ to $I$ extends uniquely to an $A$-strictly continuous homomorphism $\bar{\pi_I}:M_A(I)\to M_A(B)$.
\end{corollary}

\begin{proof}
Apply Lemma~\ref{extend C multiplier} with 
$A,I$ playing the role of $C,A$, 
with the canonical homomorphism $\rho:A\to M(I)$ playing the role of $\kappa$,
with $D=A$ and $\sigma=\lambda=\pi$,
and with $\pi_I$ playing the role of $\pi$.
Since
\[
\pi_I(I)\subset B\subset M_A(B),
\]
the hypotheses of the lemma are satisfied.
\end{proof}

\subsection*{$C$-multiplier bimodules}

The preceding subsection was really only a prelude for the current one, where we generalize the $C$-multiplier bimodules of \cite[Section~1.4]{EKQR} for tensor products.


In addition to the preceding nondegenerate homomorphism $\kappa:C\to M(A)$,
suppose
that we also have a nondegenerate $A$-correspondence $X$, and let $\varphi_A:A\to \LL(X)$ be the associated homomorphism.
Then we can compose to get a nondegenerate homomorphism
\[
\varphi_C:=\bar{\varphi_A}\circ\kappa:C\to \LL(X),
\]
where $\bar{\varphi_A}$ denotes the canonical extension of $\varphi_A:A\to \LL(X)$ to $M(A)$,
so $X$ becomes a nondegenerate $C-A$ correspondence.
Similarly on the right, so $X$ becomes a nondegenerate $C$-bimodule (but \emph{not} a correspondence, because the $A$-valued inner product cannot be turned into a $C$-valued one).
Since $M(X)$ is an $M(A)$-correspondence, by composing with $\bar\kappa$ we get an $M(C)$-bimodule structure on $M(X)$ (but again, not a correspondence).


\begin{definition}
\label{relative multiplier}
Let $X$ be a nondegenerate $A$-correspondence,
and let $\kappa:C\to M(A)$ be a nondegenerate homomorphism.
\begin{enumerate}
\item A \emph{$C$-multiplier} of $X$ is a multiplier $m\in M(X)$ such that
\[
\kappa(C)\cdot m\cup m\cdot \kappa(C)\subset X,
\]
and $M_C(X)$ denotes the set of all $C$-multipliers of $X$.

\item 
The \emph{$C$-strict} topology on $M_C(X)$ is generated by the seminorms
\[
m\mapsto \|\kappa(c)\cdot m\|
\midtext{and}
m\mapsto \|m\cdot \kappa(c)\|
\midtext{for}c\in C.
\]
\end{enumerate}
\end{definition}




The $C$-multiplier algebra $M_C(A)$ from the preceding subsection is the special case of $M_C(X)$ where $X=A$ regarded as an $A$-correspondence in the usual way.

The following generalizes \cite[Lemma~1.40]{EKQR}:

\begin{lemma}
\label{strict correspondence}
\

\begin{enumerate}
\item The $C$-strict topology is stronger than the relative strict topology on $M_C(X)$.

\item $M_C(X)$ is an $M_C(A)$-correspondence with respect to the restrictions of the operations of the $M(A)$-correspondence $M(X)$, and the operations are separately $C$-strictly continuous.

\item $\KK(M_C(X))\subset M_C(\KK(X))$.

\item $M_C(X)$ is the $C$-strict completion of $X$.

\item $M_C(X)$ is an $M(C)$-subbimodule of $M(X)$.
\end{enumerate}
\end{lemma}

\begin{proof}
Many of the arguments are quite similar to those in the proof of Lemma~\ref{C multiplier}.

(1)
If $m_i\to 0$ $C$-strictly in $M_C(X)$, we must show that if $T\in \KK(X)$ and $a\in A$ then both $\|Tm_i\|$ and $\|m_i\cdot a\|$ tend to $0$, and the argument is similar to Lemma~\ref{C multiplier}.

(2)
If $n\in M_C(A)$ and $m\in M_C(X)$ then $n\cdot m$ is in $M(X)$ because $M(X)$ is an $M(A)$-correspondence; we must show that it is in $M_C(X)$.
For $c\in C$, note that $\kappa(c) n\in A$, so because $\kappa:C\to M(A)$ is nondegenerate we can factor $\kappa(c) n=a\kappa(c')$ with $a\in A$ and $c'\in C$, and then we have
\begin{align*}
\kappa(c)\cdot (n\cdot m)
&=(\kappa(c) n)\cdot m
=(a\kappa(c'))\cdot m
=a\cdot (\kappa(c')\cdot m)\in X
\end{align*}
because $\kappa(c')\cdot m\in X$.

On the other hand, we have
\[
(n\cdot m)\cdot \kappa(c)=n\cdot (m\cdot \kappa(c))\in X
\]
because $m\cdot \kappa(c)\in X$ and $n\in M(A)$.

The $M(A)$-valued inner product on $M(X)$ restricts to one on
$M_C(X)$; we must show that for $m,n\in M_C(X)$ the inner product $\<m,n\>$ is in $M_C(A)$. For $c\in C$ we have
\begin{align*}
\<m,n\>\kappa(c)
&=\<m,n\cdot \kappa(c)\>
\in A,
\end{align*}
because $n\cdot \kappa(c)\in X$ and $\<M(X),X\>\subset A$, and similarly $\kappa(c)\<m,n\>\in A$.

The separate continuity is similar to Lemma~\ref{C multiplier}.
For example, if $m_i\to 0$ $C$-strictly in $M_C(X)$, $n\in M_C(X)$, and $c\in C$, then
\[
\|\<m_i,n\>\kappa(c)\|
=\|\<m_i,n\cdot \kappa(c)\>\|
\to 0
\]
because $n\cdot \kappa(c)\in X$, $m_i\to 0$ strictly by (1), and the operations in $M(X)$ are separately strictly continuous,
while
\[
\|\kappa(c)\<m_i,n\>\|
=\|\<m_i\cdot \kappa(c^*),n\>\|
\to 0
\]
because $\|m_i\cdot \kappa(c^*)\|\to 0$.

(3)
Here the only issue is keeping straight where everything is: we have
\[
M_C(X)\subset M(X)\subset \LL(A,X),
\]
and for $m,n\in M_C(X)$ the rank-one operator $\theta_{m,n}=mn^*$
is therefore an element of $\LL(X)=M(\KK(X))$.
We need to know that $mn^*$ is a $C$-multiplier of $\KK(X)$. So, take $c\in C$.
Since $m\in M_C(X)$ we have 
\[
\phi_C(c)m\in X=\KK(A,X),
\]
so
\[
\phi_C(c)mn^*\in \KK(X).
\]
Similarly, $\phi_C(c^*)n\in \KK(X,A)$,
so $n^*\Phi_C(c)\in \KK(A,X)$, and hence
\[
mn^*\Phi_C(c)\in \KK(X).
\]
Thus we have shown that $mn^*\in M_C(\KK(X))$.

(4) and (5)
are  similar to Lemma~\ref{C multiplier}.
For example,
\[
\bigl(\bar\kappa(M(C))\cdot M_C(X)\bigr)\cdot \kappa(C)
\subset \bar\kappa(M(C))\cdot X\subset X,
\]
because $C$ acts nondegenerately on $X$,
while
\[
\kappa(C)\cdot \bigl(\bar\kappa(M(C))\cdot M_C(X)\bigr)
\subset \kappa(C)\cdot M_C(X) \subset X.
\qedhere
\]
\end{proof}

\begin{remark}\label{KM in MK}
It will be useful to explicitly note the following special case of (3) above:
taking $C=A$ and $\kappa=\id_A$ we get
\[
\KK(M_A(X))\subset M_A(\KK(X)).
\]
Note that
\[
M_A(X)=\{m\in M(X):\phi_A(a)m\in X\text{ for all }a\in A\},
\]
so that, unlike $M(X)$, the set $M_A(X)$ depends upon the left module map $\phi_A$.
\end{remark}

The following generalizes \cite[Proposition~1.42]{EKQR}:

\begin{proposition}
\label{extend}
Let $X$ and $Y$ be nondegenerate correspondences over $A$ and $B$, respectively,
let $\kappa:C\to M(A)$ and $\sigma:D\to M(B)$ be nondegenerate homomorphisms,
let $(\psi,\pi):X\to M_D(Y)$ be a correspondence homomorphism,
and let $\lambda:C\to M(\sigma(D))$ be a nondegenerate homomorphism,
such that
\[
\pi(\kappa(c)a)=\lambda(c)\pi(a)\midtext{for}c\in C,a\in A.
\]
Then $(\psi,\pi)$ extends uniquely to a $C$-strict to $D$-strictly continuous correspondence homomorphism $(\bar\psi,\bar\pi)$ from the $M_C(A)$-correspondence $M_C(X)$ to the $M_D(B)$-correspondence $M_D(Y)$.
\end{proposition}

\begin{proof}
We take $\bar\pi:M_C(A)\to M_D(B)$ as in Lemma~\ref{extend C multiplier}.
The unique existence of $\bar\psi$ is proved similarly to Lemma~\ref{extend C multiplier}:
one shows first that $\psi$ is $C$-strict to $D$-strictly continuous,
so that by completeness there is a unique continuous linear extension,
which must be a correspondence homomorphism by
separate continuity of the operations.
\end{proof}

\begin{remark}
Again, the above proposition could be the start of an investigation of what one might call ``decorated correspondences'', where $C$ decorates the $A$-correspondence $X$, etc.
\end{remark}

Here is the Toeplitz version of Proposition~\ref{extend}:

\begin{corollary}
\label{toeplitz}
Let $X$ be a nondegenerate $A$-correspondence,
let $\kappa:C\to M(A)$ and $\sigma:D\to M(B)$ be nondegenerate homomorphisms,
let $(\psi,\pi):X\to B$ be a Toeplitz representation,
and let $\lambda:C\to M(\sigma(D))$ be a nondegenerate homomorphism,
such that
\[
\pi(\kappa(c)a)=\lambda(c)\pi(a)\midtext{for}c\in C,a\in A.
\]
Then $(\psi,\pi)$ extends uniquely to a
$C$-strict to $D$-strictly continuous Toeplitz representation
$(\bar\psi,\bar\pi)$
from the $M_C(A)$-correspondence $M_C(X)$
to $M_D(B)$.
\end{corollary}

\begin{proof}
Apply Proposition~\ref{extend} with $Y=B$, regarded as a $B$-correspondence in the usual way.
\end{proof}


\begin{corollary}
\label{C=A toeplitz}
Let $X$ be a nondegenerate $A$-correspondence,
and let $(\psi,\pi):X\to B$ be a Toeplitz representation
with $\pi$ nondegenerate.
Then:
\begin{enumerate}
\item $(\psi,\pi)$ extends uniquely to an
$A$-strictly continuous Toeplitz representation
$(\bar\psi,\bar\pi)$
from the $M(A)$-correspondence $M_A(X)$
to $M_A(B)$.

\item $\psi^{(1)}:\KK(X)\to B$
extends uniquely to an $A$-strictly continuous homomorphism
$\bar{\psi^{(1)}}:M_A(\KK(X))\to M_A(B)$.

Moreover, we have:
\begin{enumerate}
\item $\bar{\psi^{(1)}}(n\cdot m)=\bar\pi(n)\bar{\psi^{(1)}}(m)$
and $\bar{\psi^{(1)}}(m\cdot n)=\bar{\psi^{(1)}}(m)\bar\pi(n)$
for all $n\in M(A)$ and $m\in M_A(\KK(X))$;

\item $\bar{\psi^{(1)}}(mn^*)=\bar\psi(m)\bar\psi(n)^*$
for all $m,n\in M_A(X)$.
\end{enumerate}
\end{enumerate}
\end{corollary}

\begin{proof}
(1)
Apply Corollary~\ref{toeplitz} with $C=D=A$, $\kappa=\lambda=\id_A$, and $\sigma=\pi$.

(2)
First we need to justify the meaning of (b): for $m,n\in M_A(X)$ we need to know that $mn^*\in M_A(\KK(X))$.
But this follows easily from work we have already done: since $X$ is a $\KK(X)-B$ imprimitivity bimodule, arguments similar to those in the proof of Lemma~\ref{strict correspondence} (2) show that the left-hand inner product on $M_A(X)$ takes values in $M_A(\KK(X))$.

To prove (2), we apply Corollary~\ref{extend C multiplier}
with $C,A,B,D$ replaced by $A,\KK(X),B,A$, respectively;
recall that for $a\in A$ and $T\in \KK(X)$ we have
\[
\psi^{(1)}(\varphi(a)T)=\pi(a)\psi^{(1)}(T).
\]

Now everything in (2) except for (b) follows from Corollary~\ref{extend C multiplier}.
Then (b) follows from separate continuity of the operations, density of $A$, $X$, and $\KK(X)$ in $M(A)$, $M_A(X)$, and $M_A(\KK(X))$, respectively, and the corresponding property of $\psi^{(1)}$.
\end{proof}

\begin{example}
\label{degenerate compact}
We should justify our need for (2) in above corollary: if $(k_X,k_A)$ is the canonical Toeplitz representation of the correspondence $X=\HH(E)$ associated to a topological graph $E$, then $k_X^{(1)}:\KK(X)\to C^*(E)$ can be degenerate.
For example, let $E$ be the directed graph in Example~\ref{degenerate}. Then
\[
k_X^{(1)}(\KK(X))=\smtx{*&0\\0&0},
\]
so
\[
k_X^{(1)}(\KK(X))C^*(E)=\smtx{*&0\\0&0}M_2=\smtx{*&*\\0&0},
\]
and so $k_X^{(1)}$ is degenerate.
\end{example}

\begin{remark}
Although we will have no use for it in the main body of this paper, we point out that all of the above can be done with an ``asymmetric'' version of the $C$-multiplier bimodules; namely allowing correspondences with different left and right coefficient algebras. It takes no extra effort to establish the more general concepts and results.
Since we feel that they may be useful elsewhere,
we record here the asymmetric versions:

\begin{definition}
Let $X$ be a nondegenerate $A-B$ correspondence,
and let $\kappa_C:C\to M(A)$ and $\kappa_D:D\to M(B)$ be nondegenerate homomorphisms.
\begin{enumerate}
\item A \emph{$C,D$-multiplier} of $X$ is a multiplier $m\in M(X)$ such that
\[
\kappa_C(C)\cdot m\cup m\cdot \kappa_D(D)\subset X,
\]
and $M_{C,D}(X)$ denotes the set of all $C,D$-multipliers of $X$.

\item 
The \emph{$C,D$-strict} topology on $M_{C,D}(X)$ is generated by the seminorms
\[
m\mapsto \|\kappa_C(c)\cdot m\|
\midtext{and}
m\mapsto \|m\cdot \kappa_D(d)\|
\midtext{for}c\in C,d\in D.
\]
\end{enumerate}
\end{definition}

\begin{lemma}
\

\begin{enumerate}
\item The $C,D$-strict topology is stronger than the relative strict topology on $M_{C,D}(X)$.

\item $M_{C,D}(X)$ is an $M_C(A)-M_D(B)$ correspondence with respect to the restrictions of the operations of the $M(A)-M(B)$ correspondence $M(X)$, and the operations are separately continuous for the $C,D$-strict, $C$-strict, and $D$-strict topologies.

\item $M_{C,D}(X)$ is the $C,D$-strict completion of $X$.

\item $M_{C,D}(X)$ is an $M(C)-M(D)$ subbimodule of $M(X)$.
\end{enumerate}
\end{lemma}

\begin{proposition}
Let $X$ and $Y$ be nondegenerate $A-B$ and $P-Q$ correspondences, respectively,
let $\kappa_C:C\to M(A)$,
$\kappa_D:D\to M(B)$,
$\kappa_R:R\to M(P)$, and
$\kappa_S:S\to M(Q)$ be nondegenerate homomorphisms,
let ${}_\pi\psi_\theta:{}_AX_B\to {}_{M_R(P)}M_{R,S}(Y)_{M_S(Q)}$
be a correspondence homomorphism,
and let $\lambda_C:C\to M(\kappa_R(R))$
and $\lambda_D:D\to M(\kappa_S(S))$ be nondegenerate homomorphisms,
such that for $c\in C$, $a\in A$, $b\in B$, and $d\in D$ we have
\[
\pi(\kappa_C(c)a)=\lambda_C(c)\pi(a)
\midtext{and}
\theta(\kappa_D(d)b)=\lambda_D(d)\theta(b).
\]
Then ${}_\pi\psi_\theta$ extends uniquely to a
$C,D$-strict to $R,S$-strictly continuous correspondence homomorphism
${}_{\bar\pi}\bar\psi_{\bar\theta}$ from the
$M_C(A)-M_D(B)$ correspondence $M_{C,D}(X)$ to the
$M_R(P)-M_S(Q)$ correspondence $M_{R,S}(Y)$.
\end{proposition}
\end{remark}
\end{appendix}

\newcommand{\etalchar}[1]{$^{#1}$}
\providecommand{\bysame}{\leavevmode\hbox to3em{\hrulefill}\thinspace}
\providecommand{\MR}{\relax\ifhmode\unskip\space\fi MR }
\providecommand{\MRhref}[2]{%
  \href{http://www.ams.org/mathscinet-getitem?mr=#1}{#2}
}
\providecommand{\href}[2]{#2}

\end{document}